\documentclass{amsart}
\usepackage{pgf,pgfarrows,pgfnodes,pgfautomata,pgfheaps,pgfshade,hyperref, amssymb,enumerate,amsmath}
\usepackage[all]{xy}
\usepackage[capitalize]{cleveref}
\usepackage{mathtools}
\usepackage[shortlabels]{enumitem}
\usepackage{stmaryrd}
\usepackage{tcolorbox}

\definecolor{verde}{cmyk}{0.7,0,0.8,0.1}

\newtheorem{theorem}{Theorem}[section]
\newtheorem*{theorem*}{theorem}
\newtheorem{proposition}[theorem]{Proposition}
\newtheorem{lemma}[theorem]{Lemma}
\newtheorem{corollary}[theorem]{Corollary}

\newtheorem{remark}[theorem]{Remark}

\newtheorem{definition}[theorem]{Definition}

\usepackage{euscript,epsfig}
\usepackage{tikz}
\usetikzlibrary{graphs}
\usetikzlibrary[graphs]
\usetikzlibrary{arrows}
\usepackage{tikz-cd}
\usetikzlibrary{decorations.markings}
\usepackage[colorinlistoftodos]{todonotes}

\def\Z{{\mathbb Z}}

\def\N{{\mathbb N}}

\def\F{{\mathbb F}}

\def\cP{{\mathcal P}}

\def\cM{{\mathcal M}}

\def\diam{{\rm diam}}
\def\Per{{\rm Per}}
\def\Ker{{\rm Ker}}
\author[]{Paulina Cecchi-Bernales}
\address{Departamento de Matem\'aticas, Facultad de Ciencias, Universidad de Chile. Campus Juan G\'omez Millas. Las Palmeras 3425, \~Nu\~noa, Chile.}
\email{pcecchi@uchile.cl}

\thanks{P. Cecchi-Bernales was supported by ANID/Proyecto Fondecyt Postdoctorado No. 3210746 and ANID/ECOS 210033}

\author[]{Mar\'{\i}a Isabel Cortez}
\address{Facultad de Matem\'aticas, Pontificia Universidad Cat\'olica de Chile. Edificio Rolando Chuaqui, Campus San Joaquín. Avda. Vicuña Mackenna 4860, Macul, Chile.}
\email{maria.cortez@mat.uc.cl}

\thanks{The research of M.I. Cortez was supported by proyecto Fondecyt Regular No. 1190538.}

\author[]{Jaime G\'omez}
\address{Facultad de Matem\'aticas, Pontificia Universidad Cat\'olica de Chile. Edificio Rolando Chuaqui, Campus San Joaquín. Avda. Vicuña Mackenna 4860, Macul, Chile.}
\email{jagomez7@uc.cl}

\thanks{J. Gómez was supported by ANID/ Doctorado Nacional No. 21200054}

\thanks{}

\begin{document}

\title{Invariant measures of Toeplitz subshifts on  non-amenable groups }



\date{\today}

\begin{abstract}
Let $G$ be a  countable residually finite group (for instance $\F_2$) and let $\overleftarrow{G}$ be a totally disconnected metric compactification of $G$ equipped with the action of $G$ by left multiplication. For every $r\geq 1$ we construct a Toeplitz $G$-subshift $(X,\sigma,G)$, which is an almost one-to-one extension of $\overleftarrow{G}$, having $r$ ergodic measures $\nu_1, \dots,\nu_r$ such that for every $1\leq i\leq r$ the measure-theoretic dynamical system $(X,\sigma,G,\nu_i)$ is isomorphic to $\overleftarrow{G}$ endowed with the Haar measure. The construction we propose is general (for amenable and non amenable residually finite groups), however,  we point out the differences and obstructions that could appear when the acting group is not amenable. \end{abstract}

\maketitle

\section{Introduction}\label{sec:intro}

A classical problem in topological dynamics is to describe the set of invariant measures of a given dynamical system, which represents all possible measure-theoretic dynamical systems housed by it. This set is known to be a Choquet simplex whose extreme points correspond to the ergodic measures of the system, and it is known to be non-empty provided the phase space is compact and the acting group is amenable (see \cite{Gl}).  Indeed, amenable groups can be characterized as those whose continuous actions on the Cantor set always admit invariant measures \cite{GDeLaH97}. Recently, in \cite{FJ21} has been shown that the amenability of a group can be tested within the class of subshifts, i.e, the symbolic continuous actions on the Cantor set.   Actually, the existence of subshifts of non-amenable groups without invariant measures was known before \cite{FJ21}. Let us mention for instance the example of a minimal shift of finite type of the free group $\F_2$ constructed in \cite[Section 9]{BSS22}. More generally, as a consequence of \cite[Theorem 1.2]{Sew14} it can be deduced that for any non-amenable group $G$ there exists a $G$-subshift on $n$ letters ($n$ depending on $G$) which admits no invariant probability measure. To the best of our knowledge, \cite{FJ21} is the first article to say explicitly that subshifts are a test family for amenability. Let us also emphasise that the systems without invariant measures constructed in \cite{FJ21} are subshifts in a two-letter alphabet.
Considering the previous facts, it is reasonable to wonder about the existence of invariant measures for specific classes of subshifts of non-amenable groups.\\

The problem of describing the set of invariant measures of dynamical systems given by the actions of amenable groups, has been extensively explored for the class of {\it Toeplitz subshifts}. These subshifts were introduced by Jacobs and Keane in the context of $\Z$-actions \cite{JK69} and characterized as the minimal symbolic almost 1-1 extensions of odometers in \cite{DL98}. In \cite{Williams}, Williams  demonstrated that this class of $\Z$-minimal subshifts can have any finite number of ergodic measures, or countably many, or uncountably many of them. Later, Downarowicz \cite{Do91} generalized William's results by showing that any Choquet simplex can be realized as the set of invariant measures of a Toeplitz subshift in $\{0,1\}^{\Z}$, up to affine homeomorphism. Toeplitz $\Z$-subshifts have since been widely studied for the rich variety of possible phenomena they can exhibit, both regarding topological and ergodic-theoretic properties (see for instance \cite{Dow05}, \cite{Iwa96}, \cite{DL96}, \cite{GJ00}, \cite{su01}, \cite{CBD22}).  \\

Toeplitz subshifts have been generalized from $\Z$-actions to more general group actions in \cite{Cor06}, \cite{MICPET08}, \cite{Kr07} and \cite{Krieger}. One of the consequences of this generalization
is the characterization of residually finite groups as those ones that admit (non periodic) Toeplitz subshifts  \cite{MICPET08, Kr07}. In \cite{MICPET08} the authors broadened the characterization of Toeplitz subshifts as the mi\-ni\-mal symbolic almost 1-1 extensions of odometers to actions of residually finite groups.   Downarowicz's result about the realization of Choquet simplices as sets of invariant measures of Toeplitz subshifts was extended in \cite{CP14} to the class of residually finite amenable groups. This realization result has also been proved in \cite{CeCo19} for a special class of amenable groups, namely {\em congruent monotileable} amenable groups, that contains as a proper subset the class of residually finite amenable groups. As the geometry of the set of invariant measures  of a topological dynamical system is an invariant under topological orbit equivalence (\cite{GMPS10}), the problem of realizing Choquet simplices among the actions of a prescribed group, can be interpreted as a way to estimate the size of the family of topological dynamical systems given by the actions of this group up to orbit equivalence.\\

Results about the existence of invariant measures for continuous actions of non amenable groups can be found, for example, in \cite{HM06} and \cite{El21}. In this paper we focus on the existence of invariant measures of Toeplitz subshifts given by actions of countable residually finite groups which are not necessarily amenable. This encompasses all countable residually finite groups having an isomorphic copy of the free group on $n\geq 2$ generators, $\F_n$. As far as we know, the previous studies concerning the set of invariant measures of Toeplitz subshifts, strongly rely on the structure of $\mathbb{Z}$ or on the existence of F\o lner sequences, which do not exist in the case of non amenable groups.  The lack of F\o lner sequences  is one of the main challenges addressed in this work. 
Our main results are the following.

\begin{theorem}\label{theo:main0}
Let $G$  be a countable residually finite group and let $\overleftarrow{G}$ be a totally disconnected metric compactification of $G$ equipped with the action of $G$ by left multiplication. Then   there exists a  regular  Toeplitz $G$-subshift   whose maximal equicontinuous factor is $\overleftarrow{G}$.  
\end{theorem}

 Theorem \ref{theo:main1} is a direct consequence of  Theorem \ref{theo:main0}.

\begin{theorem}\label{theo:main1}
    Let $G$  be a countable residually finite group and let $\overleftarrow{G}$ be a totally disconnected metric compactification of $G$ equipped with the action of $G$ by left multiplication. Then there exists a uniquely ergodic  Toeplitz $G$-subshift $(X, \sigma,G)$ and an almost 1-1 factor map $\pi:X\to \overleftarrow{G}$, such that if $\nu$ is  the unique ergodic probability measure of $(X, \sigma, G)$ then $\pi$ is  a measure conjugacy between $(X, \sigma, G, \nu)$  and $\overleftarrow{G}$  endowed with the Haar measure.
\end{theorem}

\begin{theorem} \label{theo:main2}
Let $G$  be a countable residually finite group and let $\overleftarrow{G}$ be a totally disconnected metric compactification of $G$ equipped with the action of $G$ by left multiplication. For every integer $r>1$ there exists a Toeplitz $G$-subshift $X\subseteq \{1,\dots, r\}^G$ with at least $r$ ergodic probability measures $\nu_1,\dots, \nu_r$, and whose maximal equicontinuous factor is $\overleftarrow{G}$.  Furthermore, for every $1\leq i \leq r$, we have the following:
\begin{enumerate}
\item $(X,\sigma, G, \nu_i)$ is measure conjugate to   $\overleftarrow{G}$  endowed with the Haar measure.
\item   $\nu_i(\{x\in X: x(1_G)=i\})\geq \mu(\{x\in X: x(1_G)=i\} )$  for every invariant probability measure $\mu$.
\end{enumerate}
 \end{theorem}

The document is organized as follows. In Section \ref{sec:basicdef} we give basic notions concerning topological and measure-theoretic dynamical systems, as well as the basic background on residually finite groups, $G$-odometers and Toeplitz subshifts. Section \ref{sec:regular} is devoted to prove the existence of uniquely ergodic Toeplitz subshifts for arbitrary residually finite groups (Theorem \ref{theo:main1}). In particular, we give a necessary and sufficient condition for a Toeplitz array to be regular. In Section \ref{sec:finitely-many} we define a sequence of periodic measures of the full shift whose accumulation points are supported in a Toeplitz subshift $X$ that we introduce at the beginning of the section. We use this to show that $X$ has at least $r$ ergodic measures.
Finally, Section \ref{sec:measure-conjugate} is devoted to complete the proof  of Theorem \ref{theo:main2}. 

\section{Preliminaries}\label{sec:basicdef}

\subsection{Topological dynamical systems and invariant measures}\label{subsec:defrfgroups}
Let $G$ be a countable discrete infinite group. We denote by $1_G$ the identity of $G$. By a {\it topological dynamical system} we mean a continuous (left) action $\phi$ of  $G$ on a compact metric space $X$. We denote this topological dynamical system by $(X,\phi,G)$. We say that $(X,\phi,G)$ is {\it free} if $\phi^gx=x$ implies $g=1_G$, for every $x\in X$. The system is {\it minimal} if for every $x\in X$ its {\em orbit} $O_{\phi}(x)=\{\phi^gx: g\in G\}$ is dense in $X$. The system is {\it equicontinuous} if the collection of maps $\{\phi^g\}_{g\in G}$ is equicontinuous.  If $X$ is a Cantor set, we say that the topological dynamical system is a {\it Cantor system}.

An {\it invariant measure} of the topological dynamical system  $(X,\phi,G)$, is a probability measure $\mu$ defined on the Borel sigma-algebra of $X$, that verifies  $\mu(\phi^gA)=\mu(A)$, for every $g\in G$ and every Borel set $A$. An invariant measure $\mu$ is said to be {\em ergodic} if $\mu(A)\in\{0,1\}$ whenever $A$ is a Borel set verifying $\phi^g(A)=A$ for all $g\in G$. The set of all invariant measures admitted by a system $(X,\phi, G)$ is denoted $\cM(X,\phi,G)$. It is known to be a Choquet simplex whose extreme points correspond to the ergodic measures of $(X,\phi,G)$ (We refer to \cite[Chapter 4]{Gl} for details). When the set $\cM(X,\phi, G)$ is a singleton, the system is said to be {\em uniquely ergodic}.
If $\mu$ is an invariant measure of $(X,\phi, G)$, the quadruple $(X,\phi,G,\mu)$ is called a {\em probability-measure-preserving (p.m.p) dynamical system}.

A {\it factor map} from  $(X,\phi,G)$ to the topological dynamical system $(Y,\varphi,G)$ is a continuous surjective map $\pi:X\to Y$ such that $\pi(\phi^gx)=\varphi^g\pi(x)$, for every $x\in X$ and $g\in G$. In this case we say that $(X,\phi,G)$ is an {\it extension} of $(Y,\varphi,G)$ and $(Y,\varphi,G)$ is a {\it factor} of $(X,\phi,G)$. The factor map $\pi$ is {\it almost one to one} (or {\em almost 1-1}) if the set of points in $Y$ having only one preimage is residual. If the system $(Y,\varphi,G)$ is minimal, then this is equivalent to the existence of $y\in Y$ such that $|\pi^{-1}\{y\}|=1$. If $\pi$ is an almost 1-1 factor map, then we say that $(X,\phi,G)$ is an {\it almost 1-1 extension} of $(Y,\varphi,G)$.

An equicontinuous system $(Y,\varphi, G)$ is said to be {\it the maximal equicontinuous factor} of $(X,\phi, G)$ if there is a factor map $\pi\colon X\to Y$  such that for any other map $f\colon X\to Y'$, with $(Y',\varphi',G)$ equicontinuous, there exists a factor map $q\colon Y\to Y'$ that satisfies $q\circ \pi=f$. Moreover, if $(X,\phi,G)$ is a minimal almost $1$-$1$ extension of a minimal equicontinuous system $(Y,\varphi,G)$, then $(Y,\varphi,G)$ is the maximal equicontinuous factor of $(X,\phi,G)$ (See \cite[Proposition 5.5.]{Krieger}).

Two p.m.p dynamical systems $(X,\phi, G, \mu)$ and $(Y,\varphi, G,\nu)$ are {\em measure conjugate} if: (i) there exist conull sets $X'\subseteq X$ and $Y'\subseteq Y$ satisfying $\phi^g(X')\subseteq X'$ and $\varphi^g(Y')\subseteq Y'$ for all $g\in G$, and (ii) there is a bijective map $f:X'\to Y'$ which verifies $f, f^{-1}$ are both measurable, $\nu(A)=\mu(f^{-1}(A))$ for every measurable set $A\subseteq Y'$ and $f(\phi^g(x))=\varphi^g(f(x))$ for all $x\in X'$ and $g\in G$. In this case, we say that $f$ is a {\em measure conjugacy}.

\subsection{$G$-odometers, residually finite groups and compactifications.}
A countable group $G$ is said to be {\it residually finite}, if there exists a nested sequence   of finite index subgroups of $G$ with trivial intersection. This is equivalent to the existence of a sequence of normal subgroups with the same characteristics (see \cite{CC10} for more details about residually finite groups). Suppose that $G$ is an infinite residually finite group and let $(\Gamma_n)_{n\in\mathbb{N}}$ be a nested sequence of   finite index subgroups of $G$ such that  $\bigcap_{n\in\mathbb{N}}\Gamma_n=\{1_G\}$. The \textit{$G$-odometer} associated to $(\Gamma_n)_{n\in\mathbb{N}}$ is defined as
 \begin{align*}
    \overleftarrow{G}:=\varprojlim(G/\Gamma_n,\varphi_n)=\{(x_n)_{n\in\mathbb{N}}\in\prod_{n\in\mathbb{N}}G/\Gamma_n: \varphi_n(x_{n+1})=x_n,\mbox{ for every }n\in\mathbb{N}\},
 \end{align*}
where $\varphi_n\colon G/\Gamma_{n+1}\to G/\Gamma_n$ is the canonical projection, for every $n\in\mathbb{N}$. The space $\overleftarrow{G}$ is a Cantor set if we endow every $G/\Gamma_n$ with the discrete topology,  $\prod_{n\in\mathbb{N}}G/\Gamma_n$ with the product topology and  $\overleftarrow{G}$ with the induced topology.  Observe that when the groups $\Gamma_n$'s are normal,  $\overleftarrow{G}$ is a subgroup of $\prod_{n\in\mathbb{N}}G/\Gamma_n$, and $G$ can be seen as a dense subgroup of  $\overleftarrow{G}$  identifying $g\in G$ with  $(g\Gamma_n)_{n\in\mathbb{N}}\in \overleftarrow{G}$.
 There is a natural action $\phi$  of $G$ on $\overleftarrow{G}$ by coordinatewise left multiplication. The topological dynamical system $(\overleftarrow{G},\phi,G)$ is a free equicontinuous minimal Cantor  system which is also known as the $G$-odometer associated to the sequence $(\Gamma_n)_{n\in \N}$.  It is important to note that the $G$-odometers associated to subsequences of $(\Gamma_n)_{n\in\mathbb{N}}$ are conjugate as dynamical systems. 

A $G$-odometer having a group structure (when the $\Gamma_n$'s are normal)  is a totally disconnected metric com\-pac\-ti\-fi\-ca\-tion of $G$, i.e,  this is a  totally disconnected  metric compact group $\overleftarrow{G}$ for which there exists an injective homomorphism $i:G\to \overleftarrow{G}$ such that $i(G)$ is dense in $\overleftarrow{G}$. Conversely, every totally disconnected  metric com\-pac\-ti\-fi\-ca\-tion of $G$ is a $G$-odometer, as the next lemma shows.
\begin{lemma}\label{odometers-compactification}
If $\overleftarrow{G}$ is a totally disconnected metric compactification of $G$,  then there exists a nested sequence $(\Gamma_n)_{n\in\mathbb{N}}$ of finite index normal subgroups of $G$ with trivial intersection, such that the $G$-odometer associated to $(\Gamma_n)_{n\in\mathbb{N}}$ is conjugate to $\overleftarrow{G}$ equipped with the action of $G$ by left multiplication. This implies that $G$ is residually finite if and only if $G$ has a totally disconnected  metric compactification.
\end{lemma}
\begin{proof}
If $G$ is residually finite, then any $G$-odometer is a totally disconnected metric compactification of $G$. Now, suppose that $\overleftarrow{G}$ is a totally disconnected metric compactification of $G$. We can identify $G$ with a dense subgroup of $\overleftarrow{G}$. Since $\overleftarrow{G}$ is compact and totally disconnected,  $\overleftarrow{G}$ is profinite,   i.e., an inverse limit of finite groups equipped with the discrete topology  (see for example \cite[Theorem 1.1.12]{RZ10}).  The metrizability of $\overleftarrow{G}$  implies that  the inverse limit that defines  $\overleftarrow{G}$ is countable  (\cite[Corollary 1.1.13 and Remark 2.6.7]{RZ10}), namely $\overleftarrow{G}=\varprojlim(G_n,\tau_n)$. 
For every $n\in \mathbb{N}$, let $\pi_n:\overleftarrow{G}\to G_n$  be the natural projection.  The group $\Gamma_n=\Ker(\pi_n)\cap G$  is a finite index subgroup of $G$. Indeed, the map $\phi_n:G_n\to G/\Gamma_n$ given by $\phi_n(a)=\{(g_j)_{j\in\mathbb{N}}\in G: g_n=a\}$ is a well defined isomorphism. Since  $\varphi_n\circ \phi_{n+1}=\phi_n\circ \tau_n$, for every $n\in\mathbb{N}$, we deduce that $\overleftarrow{G}$ is conjugate to the odometer associated with the sequence $(\Gamma_n)_{n\in\mathbb{N}}$.
\end{proof}

\medskip

\subsection{Toeplitz $G$-subshifts}\label{subsec:deftoeplitzodometers} Let  $\Sigma$ be a finite set with at least two elements. The set $$\Sigma^G=\{x=(x(g))_{g\in G}: x(g)\in \Sigma, \mbox{ for every } g\in G\}$$  is a Cantor set if we endow $\Sigma$ with the discrete topology and $\Sigma^G$ with the product topology. The {\it (left) shift action} $\sigma$ of $G$ on $\Sigma^G$ is defined as
$$
\sigma^gx(h)=x(g^{-1}h), \mbox{ for every } h, g\in G \mbox{ and } x\in \Sigma^G.
$$
The topological dynamical system $(\Sigma^G, \sigma, G)$ is a Cantor system known as the {\it full G-shift}. If $X\subseteq \Sigma^G$ is a closed $\sigma$-invariant set, we say that $X$ is a {\it subshift}.  The system $(X, \sigma|_X,G)$, given by the restriction of $\sigma$ on $X$, is also called a subshift (see for example \cite{CC10}, for more details).

\medskip

The definitions and statements written in the rest of this section can be found in \cite{MICPET08}. We include some of the proofs for the sake of completeness.

\medskip
 Let $x\in \Sigma^G$ and let $\Gamma\subseteq G$ be a subgroup of finite index. We define  

\begin{align*}
    \Per(x,\Gamma,\alpha)&=\{g\in G\colon x(\gamma g)=\alpha \mbox{ for every }\gamma\in\Gamma\}, \mbox{ for every } \alpha\in \Sigma.\\
    \Per(x,\Gamma)&=\bigcup_{\alpha\in \Sigma}\Per(x,\Gamma,\alpha).
\end{align*}
The elements of $\Per(x,\Gamma)$ are those belonging to some coset $\Gamma g $ for which $x$ restricted to $\Gamma g$ is constant.

\begin{remark}
 {\rm Observe that if $\Gamma$ is normal then $g\in \Per(x,\Gamma)$ if and only if $x(g)=x(\gamma g)=x(g\gamma)$ for every $\gamma\in \Gamma$.  }
\end{remark}

\begin{remark}{\rm  If $\Gamma$ is a normal subgroup of $G$, we use indistinctly left and right cosets.   For a subgroup which is not necessarily normal, we will specify if we use left or right cosets.  }
\end{remark}

An element $\eta\in \Sigma^G$ is a \textit{Toeplitz array} or a {\em Toeplitz element}, if for every $g\in G$ there exists a finite index subgroup $\Gamma$ of $G$ such that $g\in \Per(\eta,\Gamma)$. The finite index subgroup $\Gamma$ is a {\it group of periods of} $\eta$ if $\Per(\eta,\Gamma)\neq \emptyset$. Observe that for every $g\in G$, for every $\alpha\in\Sigma$  and every group of periods $\Gamma$ of $\eta$ we have   $g\Per(\eta,g^{-1}\Gamma g,\alpha)=\Per(\sigma^g\eta,\Gamma,\alpha)$. 
A group of periods $\Gamma$ of $\eta$ is an {\it essential group of periods of} $\eta$ if $\Per(\eta,\Gamma,\alpha)\subseteq  \Per(\sigma^g\eta,\Gamma,\alpha)$ for every $\alpha\in \Sigma$, implies   $g\in \Gamma$.

 \begin{lemma}
Let $\eta\in\Sigma^G$ be a Toeplitz array. For every group of periods $\Gamma$ of $\eta$, there exists an essential group of periods $K$ of $\eta$ such that $\Per(\eta,\Gamma,\alpha)\subseteq\Per(\eta, K, \alpha)$, for every $\alpha\in\Sigma$.
 \end{lemma}
\begin{proof}
Let $\Gamma$ be a group of periods of $\eta$. There exists a normal finite index subgroup $H$ of $G$ such that $H\subseteq \Gamma$ (see for example \cite[Lemma 1]{MICPET08}). Observe that $\Per(\eta,\Gamma,\alpha)\subseteq \Per(\eta, H, \alpha)$, for every $\alpha\in \Sigma$. Since $w\in \Per(\eta, H, \alpha)$ if and only if $Hw\subseteq \Per(\eta, H, \alpha),$ the set $\Per(\eta, H, \alpha)$ is a disjoint union of right  equivalence classes in $G/H$. Namely,    $\Per(\eta, H, \alpha)=Hw_1\cup\dots\cup Hw_m$. Thus for every $g\in G$, the set $g\Per(\eta,H,\alpha)$ is also a disjoint union of the same number of equivalence classes. This implies that if 
$\Per(\eta,H,\alpha)\subseteq g\Per(\eta,g^{-1}Hg,\alpha)=g\Per(\eta,H,\alpha)$ then $\Per(\eta,H,\alpha) =g\Per(\eta,H,\alpha)$. From this it follows that $K$, the set of all $g\in G$ such that $\Per(\eta,H,\alpha) \subseteq g\Per(\eta,g^{-1}Hg,\alpha)$ for every  $\alpha\in \Sigma$, is a group. Since $K$ contains $H$, $K$ is a finite index subgroup. Furthermore, if $w\in \Per(\eta,H,\alpha)$ then $gw\in \Per(\eta,H,\alpha)$, for every $g\in K$, which implies that $\Per(\eta, H,\alpha)\subseteq \Per(\eta, K,\alpha)$, for every $\alpha\in \Sigma$, because  $\alpha=\eta(w)=\eta(gw)$, for every $g\in K$. On the other hand, if $\Per(\eta,K,\alpha)\subseteq g\Per(\eta, g^{-1}Kg,\alpha)$, then for $w\in \Per(\eta, H, \alpha)$ we have $g^{-1}w\in \Per(\eta, g^{-1}Kg,\alpha)$, which implies $\alpha=\eta(g^{-1}w)=\eta(g^{-1}kw)$, for every $k\in K$. Since $H\subseteq K$ is a normal subgroup, in particular we have $\alpha=\eta(hg^{-1}w)$, for every $h\in H$. This implies $g^{-1}w\in \Per(\eta, H, \alpha)$ and then $g\in K$. This shows that $K$ is an essential group of periods.
\end{proof}
 
  If $\eta\in \Sigma^G$ is a Toeplitz array, then there exists a {\it period structure}  $(\Gamma_n)_{n\in\mathbb{N}}$ of $\eta$, that is, a nested sequence of essential periods of $\eta$ such that $G=\bigcup_{n\in\mathbb{N}}\Gamma_n$ (see \cite[Corollary 6]{MICPET08}).

 A \textit{Toeplitz $G$-subshift} or simply a {\em Toeplitz subshift}, is the subshift generated by the closure of the $\sigma$-orbit of a Toeplitz array.
 
Let $\eta\in \Sigma^G$ be a Toeplitz array and let $(\Gamma_n)_{n\in\mathbb{N}}$ be a period structure of $\eta$. Let $X=\overline{O_{\sigma}(\eta)}$ be the associated Toeplitz subshift. For each $n\in\mathbb{N}$, we define
 $$C_n=\{ x\in X: \Per(x,\Gamma_n, \alpha)=\Per(\eta, \Gamma_n, \alpha), \mbox{ for every } \alpha\in \Sigma\}.$$
Using the fact that every $\Gamma_n$ is an essential group of periods, it is possible to verify that $\sigma^g\eta\in C_n$ if and only if $g\in \Gamma_n$. This implies  following result.
 \begin{lemma}[{\cite[Lemma 8 and Proposition 6]{MICPET08}}]
 For every $n\in\mathbb{N}$, the set $C_n$ is closed. In addition, for every $g,h\in G$ the following statements are equivalent:
 \begin{enumerate}
\item $\sigma^gC_n\cap \sigma^hC_n\neq \emptyset.$
\item $\sigma^gC_n=\sigma^hC_n.$
\item $g\Gamma_n=h\Gamma_n$.
 \end{enumerate}
From this we get that the collection $\{\sigma^{g^{-1}}C_n: g\in D_n\}$ is a clopen partition of $X$.
\end{lemma}
\begin{proof} 
Let denote $X_n=\overline{\{\sigma^{\gamma}\eta\colon \gamma\in\Gamma_n\}}$, for each $n\in\mathbb{N}$. It is straightforward to check that  $\Per(\eta,\Gamma_n,\alpha)\subseteq \Per(x,\Gamma_n,\alpha)$, for every $x\in X_n$ and $\alpha\in \Sigma$. Let $H\subseteq \Gamma_n$ be a normal finite index subgroup of $G$. From \cite[Theorem 1.13]{Au88}, for every $x\in X$ the action of $H$ on  $\overline{\{\sigma^{h}x \colon h\in H\}}$ is minimal. Since $H$ is a finite index subgroup of $\Gamma_n$, applying again \cite[Theorem 1.13]{Au88}, we get that for every $x\in X$ the action of $\Gamma_n$ on  $\overline{\{\sigma^{\gamma}x \colon \gamma\in \Gamma_n\}}$  is minimal. In particular, 
  the action of $\Gamma_n$ on  $X_n$ is minimal. From this  we have $\Per(x,\Gamma_n,\alpha)\subseteq \Per(\eta,\Gamma_n,\alpha)$ for every $x\in X_n$ and $\alpha\in \Sigma$, which implies that $X_n\subseteq C_n$. On the other hand, let $x\in C_n$ and $(g_i)_{i\in\mathbb{N}}$ be a sequence in $G$ such that $(\sigma^{g_i}x)_{i\in\mathbb{N}}$ converges to $\eta$. 
Since $\Gamma_n$ is a finite index subgroup, taking subsequences, there exists $w\in G$ such that $g_i=w\gamma_i$ for some $\gamma_i\in\Gamma_n$, for every $i\in\mathbb{N}$. This implies that $\sigma^{w^{-1}}\eta$ is in $Y=\overline{\{ \sigma^{\gamma}x: \gamma\in \Gamma_n\}}$ and then $\Per(\eta,\Gamma_n,\alpha)\subseteq \Per(\sigma^{w^{-1}}\eta, \Gamma_n, \alpha)$ for every $\alpha\in \Sigma$.
The property of being essential implies that $w\in \Gamma_n$ and then $\eta\in X_n\cap Y$. Since $Y$ is a minimal component with respect to the action of $\Gamma_n$,   we deduce that $Y=X_n$ and then $C_n=X_n$. Observe that $X$ is a finite disjoint union of minimal components for the action of $\Gamma_n$ (If $w_1,\dots,w_k$ are representatives of each coset in $\{g\Gamma_n: g\in G\}$, then $X$ is the union of the minimal components that contain the points $\sigma^{w_1}\eta, \dots, \sigma^{w_k}\eta$), this implies that $X_n=C_n$ is clopen. From the property of being an essential period, it follows that $\sigma^g\eta\in C_n$ implies $g\in\Gamma_n$. From the minimality of the action of $G$ and the fact that $C_n$ is open,  we can infer that if $\sigma^gx\in C_n$, then $g\in\Gamma_n$, for every $x\in C_n$. This implies the results of the Lemma.
\end{proof}

\begin{remark}
{\rm  Observe that the stabilizer of $x\in X$ is the group $\bigcap_{n\in\mathbb{N}} v_n\Gamma_n v_n^{-1}$, where $v_n\in G$ is such that $x\in \sigma^{v_n}C_n$. Thus if the groups $\Gamma_n$ are normal subgroups of $G$, the Toeplitz subshift $(X,\sigma, G)$ is free if and only if $\bigcap_{n\in\mathbb{N}}\Gamma_n=\{1_G\}$.     }
\end{remark}

\begin{proposition}[{\cite[Proposition 5, Theorem 2 and Proposition 7]{MICPET08}}]\label{1-1-extension}
Let $\eta\in \Sigma^G$ be a Toeplitz array and let $X=\overline{\{\sigma^g\eta: g\in G\}}$. Suppose that $(X,\sigma|_X,G)$ is free and   $(\Gamma_n)_{n\in\mathbb{N}}$ is a period structure of $\eta$. Let $\overleftarrow{G}$ be the $G$-odometer associated to $(\Gamma_n)_{n\in\mathbb{N}}$. The map $\pi:X\to \overleftarrow{G}$ given by
$$
\pi(x)=(g_n\Gamma_n)_{n\in\mathbb{N}}, \mbox{ where  } x\in \sigma^{g_n}C_n, \mbox{ for every } n\in\mathbb{N}
$$
is an almost 1-1 factor map (then $\overleftarrow{G}$ is the maximal equicontinuous factor of $X$). Moreover,
$$
\{x\in  X: x\mbox{ is a Toeplitz array }   \}=\pi^{-1}\{y\in \overleftarrow{G}: |\pi^{-1}\{y\}|=1\}.
$$
\end{proposition}

\subsection{Good sequences of tiles for residually finite groups}\label{subsec:rflemmas}

In \cite{CP14} it was shown that amenable residually finite groups have   F\o lner sequences with additional properties. We prove in Lemmas \ref{fundamental} and \ref{decom} that when the residually finite group $G$ is not amenable, it has
sequences of finite subsets having similar properties, without being F\o lner.

\begin{lemma}\label{fundamental}
Let $G$ be a residually finite group, and let $(\Gamma_n)_{n\in \mathbb{N}}$ be a strictly decreasing sequence of normal finite index subgroups of $G$, such that $\bigcap_{n\in\mathbb{N}}\Gamma_n=\{1_G\}.$ There exists an increasing sequence $(n_i)_{i\in \mathbb{N}}\subseteq \mathbb{N}$ and $(D_i)_{i\in \mathbb{N}}$ a sequence of finite subsets of $G$ such that for every $i\in \mathbb{N}$
\begin{enumerate}
    \item $1_G\in D_1$ and $D_i\subseteq D_{i+1}$.
    \item $D_i$ is a fundamental domain of $\Gamma_{n_i}$.
    \item $G=\bigcup_{i\in \mathbb{N}} D_i$.
   \item $D_{i+1}=\bigcup_{v\in D_{i+1}\cap \Gamma_{n_i}}vD_i$. 
\end{enumerate}
\begin{proof}
Observe that if $S\subseteq G$ is a finite set, then there exists $l\geq 1$ such that if $g_1$ and $g_2$ are two different elements in $S$,  then  $g_1\Gamma_l\neq g_2\Gamma_l$. Indeed, suppose there exists a finite subset $S$ of $G$ such that for every $l\geq 1$ there are two different elements $g_l$ and $h_l$ in $S$ such that $g_l\Gamma_l=h_l\Gamma_l$.  Since $S$ is finite, there exists a subsequence $(\Gamma_{l_i})_{i\in \mathbb{N}}$ such that $g_{l_i}=g$ and $h_{l_i}=h$, for every $i\in \mathbb{N}$, where $g,h\in G$. This implies that $h^{-1}g\in \bigcap_{i\in \mathbb{N} }\Gamma_{l_i}=\bigcap_{i\geq 0}\Gamma_i=\{1_G\}$, which is a contradiction with the fact that $g\neq h$. 

 Suppose $G=\{g_1,g_2,\dots\}$ with $1_G=g_1$. We will construct a sequence $(F_i)_{i\in\mathbb{N}}$ of finite subsets of $G$ verifying points (1), (2) and (3): choose $F_1$ as a fundamental domain of $G/\Gamma_1$ such that $g_1\in F_1$. Let $\Gamma_{n_1}=\Gamma_1$. Let $S_2=F_1\cup\{g_1,g_2\}$. By the previous discussion, there exists $\Gamma_{n_2}\subseteq \Gamma_{n_1}$ satisfying $g\Gamma_{n_2}\neq h\Gamma_{n_2}$  for every pair of distinct elemetns $g,h\in S_2$. Thus, we can take $F_2$ such that $S_2\subseteq F_2$.
Suppose we have defined $\Gamma_{n_{k-1}}$ and $F_{k-1}$ such that $\Gamma_{n_{k-1}}\subseteq \Gamma_{n_{k-2}}$ and $F_{k-1}$ is a fundamental domain of $\Gamma_{n_{k-1}}$ containing the set $F_{k-2}\cup \{g_1,\dots, g_{k-1}\}$. Now, consider $S_k=\{g_1,g_2,\dots, g_k\}\cup F_{k-1}$. By the previous discussion, we can choose $\Gamma_{n_k}\subseteq \Gamma_{n_{k-1}}$ such that the elements in $S_k$ are in different equivalence classes of $G/\Gamma_{n_k}$. Therefore, we can get $F_k$ satisfying $S_k\subseteq F_k$.
  
Thus, we obtain \textit{(1)} and \textit{(2)}. Additionally, 
\begin{align*}
    G=\bigcup_{i=1}^\infty S_i\subseteq  \bigcup_{i=1}^\infty F_i,
\end{align*}
which implies \textit{(3)}.

Now we will apply \cite[Lemma 3]{CP14} in order to get a sequence $(D_i)_{i\in\mathbb{N}}$ as desired. Define   $n_{i_1}=n_1$ and $D_1=F_1$. Let $j>1$ and suppose that we have defined $n_{i_1}<\dots<n_{i_{j-1}}$ and $D_{j-1}$ a fundamental domain of $G/\Gamma_{n_{i_{j-1}}}$. Since $G=\bigcup_{v\in\Gamma_{n_{i_{j-1}}}} vD_{j-1}$, there exists $i_j>i_{j-1}$ such that 
    \begin{align*}
        F_{i_{j-1}}\subseteq \bigcup_{v\in F_{i_j}\cap\Gamma_{n_{i_{j-1}}}} vD_{j-1}.
    \end{align*}
    
Let define $D_j:=\bigcup_{v\in F_{i_j}\cap\Gamma_{n_{i_{j-1}}}} vD_{j-1}$. By construction, the sequence $(D_j)_{j\in\mathbb{N}}$ verifies (3). Applying 
\cite[Lemma 3]{CP14} we get that $(D_j)_{j\in\mathbb{N}}$ satisfies (1), (2) and (4) for the subsequence $(\Gamma_{n_{i_j}})_{j\in\mathbb{N}}$.
 \end{proof}
\end{lemma}

The next Lemma follows by induction applying Lemma \ref{fundamental}.
\begin{lemma}\label{decom}
Let $G$ be a residually finite group, and let $(\Gamma_n)_{n\in \mathbb{N}}$ be a strictly decreasing sequence of normal finite index subgroups of $G$, such that $\bigcap_{n\in\mathbb{N}}\Gamma_n=\{1_G\}.$ There exists an increasing sequence $(n_i)_{i\in \mathbb{N}}\subseteq \mathbb{N}$ and $(D_i)_{i\in \mathbb{N}}$ a sequence of finite subsets of $G$ such that for every $i\in \mathbb{N}$
\begin{enumerate}
    \item $\{1_G\}\subseteq D_i\subseteq D_{i+1}$
    \item $D_i$ is a fundamental domain of $G/\Gamma_{n_i}$.
    \item $G=\bigcup_{i=1}^\infty D_i$.
    \item $D_j=\bigcup_{v\in D_j\cap \Gamma_{n_i}} vD_i$, for each $j>i\geq 1$.
\end{enumerate}
\end{lemma}

\section{Uniquely ergodic case: regular Toeplitz subshifts}\label{sec:regular}

In this section we prove  Theorem  \ref{theo:main1}.

\subsection{Factor maps and measures.}\label{subsection-measures}

Let $(X,\phi,G)$ and $(Y,\varphi,G)$ be two topological dynamical systems such that $(Y,\varphi,G)$ admits an invariant measure $\nu$. Let $\beta(X)$ and $\beta(Y)$ the Borel sigma-algebras of $X$ and $Y$ respectively. If $\pi:X\to Y$ is a factor map then the collection of sets
$\pi^{-1}\beta(Y)=\{\pi^{-1}(A): A\in \beta(Y)\}$
is a sigma-algebra contained in $\beta(X)$. Moreover, if $A\in \pi^{-1}\beta(Y)$ then $\phi^g(A)\in \beta(Y)$, for every $g\in G$.

\begin{remark}
{\rm If $\omega$ is an invariant probability measure of $(X,\phi,G)$, then\\ $\nu^*:\beta(Y)\to [0,1]$ defined as $\nu^*(A)=\omega(\pi^{-1}(A))$ for every $A\in \beta(Y)$,  is an invariant probability measure of $(Y,\varphi,G)$. Thus, if $(Y,\varphi,G)$ is uniquely ergodic then $\nu^*=\nu$, which implies that all the invariant probability measures of $(X,\phi,G)$   coincide on $\pi^{-1}\beta(Y)$.}
\end{remark}
Since
$$\{y\in Y: |\pi^{-1}(y)|=1\}=\bigcap_{n\geq 1}\left\{y\in Y: \diam(\pi^{-1}(y))<\frac{1}{n}\right\}, 
$$
the set $\{y\in Y: |\pi^{-1}(y)|=1\}$ is a $G_{\delta}$-set (see \cite{Dow05}). Thus, the set $$\{y\in Y: |\pi^{-1}\{y\}|>1\}$$ belongs to $\beta(Y)$.

\begin{lemma}\label{lem:measurelifting}
Let $\mu:\pi^{-1}\beta(Y)\to [0,1]$ be the map defined by $\mu(\pi^{-1}(A))=\nu(A)$, for every $A\in \beta(Y)$. The map $\mu$ is an invariant probability measure on $\pi^{-1}\beta(Y)$ and if $\nu(\{y\in Y: |\pi^{-1}\{y\}|>1\})=0$, then $\mu$ extends to a unique invariant probability measure on $\beta(X)$.
\end{lemma}
\begin{proof}
It is straightforward to check that $\mu$ is an invariant probability measure on $\pi^{-1}\beta(Y)$. Let $A\subseteq X$ be a closed set. Then $\pi(A)\in \beta(Y)$, because $\pi(A)$ is compact. Thus $B=\pi^{-1}(\pi(A))$ belongs to $\pi^{-1}\beta(Y)$. Observe that $A\subseteq B$ and 
$$B\setminus A\subseteq \pi^{-1}(\{y\in Y: |\pi^{-1}\{y\}|>1\}).$$  
This implies that $B\setminus A$ is a negligible set (which implies that $B\setminus A$ is in the  completion of $\pi^{-1}\beta(Y)$ with respect to $\mu$), and since $A=B\setminus (B\setminus A)$, we get that $A$ is in the completion of $\pi^{-1}\beta(Y)$. It follows that every open set is in the completion of $\pi^{-1}\beta(Y)$, from which we get that $\beta(X)$ is contained in the completion of $\pi^{-1}\beta(Y)$. This implies that $\mu$ can be extended in a unique way to $\beta(X)$. It is straightforward to check that $\mu$ is invariant on $\beta(X)$.      
\end{proof}

\subsection{Regular Toeplitz $G$-subshifts.} Let $\Sigma$ be a finite set with at least two elements. Let $\eta\in\Sigma^{G}$ be a Toeplitz array, and suppose that the decreasing sequence of finite index  subgroups $(\Gamma_n)_{n\in \mathbb{N}}$ is a period structure of $\eta$. The Toeplitz $G$-subshift $X=\overline{O_{\sigma}(\eta)}$ is an almost 1-1 extension of the  $G$-odometer $\overleftarrow{G}$ associated to $(\Gamma_n)_{n\in\mathbb{N}}$ (see Proposition \ref{1-1-extension}).
If $\pi:X\to \overleftarrow{G}$ is the almost 1-1 factor map, then
$$
\mathcal{T}=\{x\in X: x \mbox{ is a Toeplitz array}\}=\pi^{-1}\{y\in \overleftarrow{G}: |\pi^{-1}\{y\}|=1\}. 
$$
In other words, the set of Toeplitz arrays in $X$ is exactly the pre-image of the set of elements in $\overleftarrow{G}$ having exactly one pre-image (see Proposition \ref{1-1-extension}). The set of Toeplitz arrays is invariant under the action of $G$. Indeed, if $x\in \mathcal{T}$ and $g, h\in G$, there exists a finite index subgroup $\Gamma$ of $G$ such that  $x(g^{-1}h)=x(\gamma g^{-1}h)$, for every $\gamma\in \Gamma$. From this we get 
$$\sigma^gx(h)=x(g^{-1}h)=x(\gamma g^{-1}h)=x(g^{-1}g\gamma g^{-1}h)=\sigma^gx(g\gamma g^{-1}h) \quad \forall \gamma \in \Gamma,$$ 
which implies that $h\in \Per (\sigma^gx,g\Gamma g^{-1})$. This shows that $\sigma^gx\in \mathcal{T}.$ Since $\mathcal{T}$ is invariant, its image $\pi(\mathcal{T})$ is invariant and measurable in $\overleftarrow{G}$ (see for example \cite[Theorem 2.8]{Gl}). Thus, if $\nu$ is the unique ergodic measure of $\overleftarrow{G}$, then $\nu(\pi(\mathcal{T}))\in \{0,1\}$.

\begin{proposition}\label{regular-measure}
 The following statements are equivalent:
 \begin{enumerate}
\item $\nu(\pi(\mathcal{T}))=1$.
\item There exists an invariant probability measure $\mu$ of $(X, \sigma, G)$ such that $\mu(\mathcal{T})=1$.
\item There exists a unique invariant probability measure $\mu$ of $(X,\sigma, G)$ and  $\mu(\mathcal{T})=1$.
 \end{enumerate}
\end{proposition}
\begin{proof}
If $\nu(\{y\in \overleftarrow{G}: |\pi^{-1}\{y\}|=1\})=1$ then the measure $\mu$ defined on $\pi^{-1}\beta(\overleftarrow{G})$, extends to a unique invariant probability measure on $\beta(X)$ (see Lemma \ref{lem:measurelifting}). Moreover, $\mu(\mathcal{T})=1$.  From this we get (1) implies (3). The rest of the proof is obvious.
\end{proof}
 
\begin{definition}
We say that the Toeplitz element $\eta\in \Sigma^G$ or the Toeplitz $G$-subshift $(X,\sigma,G)$ is {\it regular}, if any of the statements of Proposition \ref{regular-measure} is satisfied. 
\end{definition}
\begin{remark}
{\rm In \cite{LaSt18}, the notion of regularity over Toeplitz subshifts is defined for the case where the acting group is amenable. Using the Lemma below, it can be shown that the definition presented here coincides with that used in \cite{LaSt18} when $G$ is amenable. Note that, by definition, regular Toeplitz $G$-subshifts are uniquely ergodic.}   
\end{remark}

\medskip

Let $(D_i)_{i\in\N}$ be a sequence of fundamental domains of $G$   as in Lemma \ref{decom} (in the language of Lemma \ref{decom}, after taking subsequences, we can assume that $n_i=i$). For each $i\in \N$, let define
$$d_i=\frac{|D_i\cap \Per(\eta,\Gamma_i)|}{|D_i|}.$$
By Lemma \ref{decom}, we have that
\begin{align*}
    D_{i+1}\cap\Per(\eta,\Gamma_i)&=\left[\bigcup_{\gamma\in D_{i+1}\cap\Gamma_{i}}\gamma D_i\right]\cap \Per(\eta,\Gamma_i)\\
    &=\bigcup_{\gamma\in D_{i+1}\cap\Gamma_i}\gamma (D_i\cap\gamma^{-1}\Per(\eta,\Gamma_i)).
\end{align*}
Since $\gamma'\Per(\eta,\Gamma_i)=\Per(\eta,\Gamma_i)$ for every $\gamma'\in\Gamma_i$, we obtain that

\begin{align*}
    |D_{i+1}\cap \Per(\eta,\Gamma_i)|=|D_{i+1}\cap \Gamma_i||D_i\cap \Per(\eta,\Gamma_i)|.
\end{align*}

Since $\Per(\eta,\Gamma_i)\subseteq \Per(\eta, \Gamma_{i+1})$, 
\begin{align*}
    \frac{|D_{i+1}\cap\Per(\eta,\Gamma_{i+1})|}{|D_{i+1}|}\geq \frac{|D_{i+1}\cap\Per(\eta,\Gamma_{i})|}{|D_{i+1}|}.
\end{align*}

Now, using $|D_{i+1}\cap \Gamma_i||D_i|=|D_{i+1}|$ we conclude that

$$
d_{i+1}=\frac{|D_{i+1}\cap \Per(\eta,\Gamma_{i+1})|}{|D_{i+1}|} \geq \frac{|D_{i+1}\cap \Gamma_i||D_i\cap\Per(\eta, \Gamma_i)|}{|D_{i+1}|}=\frac{|D_i\cap\Per(\eta,D_i)|}{|D_i|}=d_i.
$$
This implies that the sequence $(d_i)_{i\in\mathbb{N}}$ is increasing, therefore, there exists $d\in [0,1]$ such that $\lim_{i\to\infty}d_i=d$.   

 \begin{lemma}\label{Periods-Toeplitz}
    For every Toeplitz array $x\in X$, we have $G=\bigcup_{n\in\mathbb{N}}\Per(x,\Gamma_n)$.
\end{lemma}
\begin{proof}
    Let $x\in X$ be a Toeplitz array and let $\{H_n\}_{n\in\mathbb{N}}$ be  a period structure of $x$. Since $X$ is minimal, we have $X=\overline{\mathcal{O}_\sigma(x)}$. Thus, Proposition \ref{1-1-extension}  implies that there exists an almost 1-1 factor map $\rho\colon X\to \overleftarrow{H}$, where $\overleftarrow{H}=\varprojlim(G/H_n,p_n)$, and $p_n\colon G/H_{n+1}\to G/H_n$ is the canonical projection for every $n\in \mathbb{N}$.  By \cite[Proposition 5.5]{Krieger}, we get that $\overleftarrow{H}$ is the maximal equicontinuous factor of $X$ and consequently, we obtain that $\overleftarrow{G}$ and $\overleftarrow{H}$ are conjugate with conjugacy $\overline{\rho}\colon \overleftarrow{G}\to\overleftarrow{H}$.  Let $\rho\colon \overleftarrow{G}\to\overleftarrow{H}$ be defined by $\rho(\overline{g})=\overline{\rho}(\overline{g})\overline{\rho}(e_{\overleftarrow{G}})$, where $\overline{g}\in\overleftarrow{G}$, $e_{\overleftarrow{G}}$ and $ e_{\overleftarrow{H}}$ denote the class of the neutral element $1_G$ in $\overleftarrow{G}$ and in $\overleftarrow{H}$, respectively.  Since $\overline{\rho}$ is a conjugacy, it implies that $\rho$ is so as well. Moreover, $\rho(e_{\overleftarrow{G}})=e_{\overleftarrow{H}}$. Applying \cite[Lemma 2]{MICPET08}, we get that for every $n\in\mathbb{N}$, there exists $k_n\in\mathbb{N}$ such that $\Gamma_{k_n}\subseteq H_{n}$ and consequently, $\Per(x,\Gamma_{n_k})\supseteq \Per(x,H_{n})$. Hence,
    \begin{align*}
        G=\bigcup_{n\in\mathbb{N}}\Per(x, H_n)\subseteq \bigcup_{k\in\mathbb{N}}\Per(x, \Gamma_{n_k})\subseteq \bigcup_{n\in\mathbb{N}}\Per(x, \Gamma_n).
    \end{align*}
    This completes the proof.
\end{proof}

\begin{lemma}\label{lem:regularmeasure}
 Let   $\eta\in \Sigma^G$ be a Toeplitz array and let $(\Gamma_n)_{n\in\mathbb{N}}$ be a period structure of $\eta$ such that every $\Gamma_n$ is normal in $G$.  Then   $\eta$ is regular if and only if $d=1$.
\end{lemma}
\begin{proof}
For every $n\in \mathbb{N}$,  let $C_n$ be as defined in Section \ref{subsec:deftoeplitzodometers}. We define
\begin{align*}
    H_n=\bigcup_{h\in \Per(\eta,\Gamma_n)}\sigma^{h^{-1}}C_n  \hspace{3mm} \mbox{ and } \hspace{3mm} H=\bigcup_{n\in\mathbb{N}}H_n.
\end{align*}
If $x\in H$, then there exist $n\in \mathbb{N}$ and $h\in \Per(\eta,\Gamma_n)$ such that $x\in \sigma^{h^{-1}}C_n$. This implies that $\Per(\sigma^{h} x,\Gamma_n,\alpha)=\Per(\eta,\Gamma_n,\alpha)$, for every $\alpha\in \Sigma$. Moreover, since $h\in \Per(\eta,\Gamma_n)$, we have $h\in \Per(\sigma^{h}x, \Gamma_n)$, which means that  $x(1_G)=x(h^{-1}\gamma h)$, for every $\gamma\in \Gamma_n$, (recall that the $\Gamma_n$'s are normal). Therefore, $1_G$  is periodic in $x$ with respect to a finite index subgroup of $G$.
 We have shown that $H_n\subseteq\{x\in X\mid 1_G\in \Per(x,\Gamma_n)\}$.  
 
 Now, let $x\in X$ such that $1_G\in\Per(x,\Gamma_n)$. Since $\{\sigma^{h^{-1}}C_n:h\in D_n\}$ is a partition of $X$, there exists $h\in D_n$ such that $x\in \sigma^{h^{-1}}C_n$. It follows that $\Per(x,\Gamma_n)=\Per(\sigma^{h^{-1}}\eta,\Gamma_n)$. Since $1_G\in \Per(x,\Gamma_n)=\Per(\sigma^{h^{-1}}\eta,\Gamma_n)$, we can deduce that $h\in \Per(\eta,\Gamma_n)$ and $x\in H_n$. Therefore, we have $H_n=\{x\in X\mid 1_G\in\Per(x,\Gamma_n)\}$.

 \vspace{1em}From above we deduce that  $\bigcap_{g\in G}\sigma^gH \subseteq \mathcal{T}$. Next, let $x\in \mathcal{T}$.  
 From Lemma \ref{Periods-Toeplitz}, we get that $G=\bigcup_{n\in\mathbb{N}}\Per(x,\Gamma_n)$. For each $g\in G$, we have that $\sigma^{g^{-1}}x$ is also a Toeplitz array. Thus, $1_G\in\Per(\sigma^{g^{-1}}x,\Gamma_i)$ for some $i\geq 1$, which implies $\sigma^{g^{-1}} x\in H$. Hence, $x\in \sigma^{g} H$ for every $g\in G$. Consequently, $x\in\bigcap_{g\in G}\sigma^g H$. We have shown $\mathcal{T}\subseteq \bigcap_{g\in G}\sigma^g H$ and with this, $\mathcal{T}=\bigcap_{g\in G}\sigma^g H$.
\vspace{1em}

Suppose $d=1$. Since $C_n=\pi^{-1}\{(y_j)_{j\in \mathbb{N}}\in \overleftarrow{G}: y_n=\Gamma_n\}$, we have $$H_n=\pi^{-1}\left(\bigcup_{h\in \Per(\eta,\Gamma_n)}\{(y_j)_{j\in \mathbb{N}}\in \overleftarrow{G}: y_n=h^{-1}\Gamma_n\}\right)$$
and
$$
H=\pi^{-1}\left( \bigcup_{n\in\mathbb{N}} \bigcup_{h\in \Per(\eta,\Gamma_n)}\{(y_j)_{j\in \mathbb{N}}\in \overleftarrow{G}: y_n=h^{-1}\Gamma_n\}   \right)
$$
This implies that $\nu(\pi(H_n))=\dfrac{|D_n\cap \Per(\eta,\Gamma_n)|}{|D_n|}=d_n$ and $\nu(\pi(H))= d=1$. On the other hand,
$$
\pi\left(\bigcap_{g\in G}\sigma^gH\right)=\bigcap_{g\in G}\phi^g\pi(H)\subseteq \pi(\mathcal {T} ). 
$$
Since $\nu$ is invariant, we deduce that $\nu(\pi(\mathcal{T}))=1$.

Conversely, suppose that $d<1$.   Since $\nu(\pi(\mathcal{T}))\leq \nu(\pi(H))=d<1$ and $\pi(\mathcal{T})$ is $G$-invariant, we obtain $\nu(\pi(\mathcal{T}))=0$. Thus, we conclude $\eta$ is not regular.
\end{proof}

\begin{proposition}\label{prop:regularexistence}
Let $G$ be a residually finite group. Let $(\Gamma_n)_{n\in \mathbb{N}}$ be a decreasing sequence of finite index normal subgroups of $G$ with trivial intersection. Let $\overleftarrow{G}$ be the $G$-odometer associated to the sequence $(\Gamma_n)_{n\in\mathbb{N} }$. Then, there exists a free regular Toeplitz $G$-subshift $X\subseteq \{0,1\}^G$  which is an almost 1-1 extension of $\overleftarrow{G}$.
\end{proposition}
\begin{proof} 
We will show that the example given in \cite[Theorem 5]{MICPET08} is regular.
Let $(\Gamma_n)_{n\in \mathbb{N}}$ be the sequence given in the statement.
Consider $(D_n)_{n\in \mathbb{N}}$ as in \cref{decom}. Taking subsequences if necessary, we can assume that $D_i$ is a fundamental domain of $G/\Gamma_i$ and $\frac{|D_{i-1}|}{|D_i|}<\frac{1}{i}$, for every $i\in \mathbb{N}$. 
Define the sequence $(S_n)_{n\geq 0}$ of subsets of $G$ inductively as follows: let $S_0:=\{1_G\}$. 
Consider $v_1\in D_1\setminus \{1_G\}$ and let $S_1:=\{v_1\}$. 
For $n>1$, suppose we have defined $S_{n-1}$ and $v_{n-1}\in S_{n-1}$. Let $S_n=(v_{n-1}\Gamma_{n-1}\cap D_n)\setminus D_{n-1}$ and let $v_n\in S_n$. The sequence $\eta\in \{0,1\}^G$ is defined as follows:
\begin{align*}
    \eta(w)=\begin{cases}
        0,&\mbox{ if }w\in\bigcup_{n\geq 0}S_{2n}\Gamma_{2n+1},\\
        1,&\mbox{ else.}
    \end{cases}
\end{align*}
The subshift  $X=\overline{O_\sigma(\eta)}$ is a Toeplitz $G$-subshift which is an almost 1-1 extension of $\overleftarrow{G}$. Moreover,  $(\Gamma_n)_{n\in \mathbb{N}}$ is a periodic structure of $\eta$ (see \cite{MICPET08}).

It remains to prove that $\eta$ is a regular Toeplitz array. Let $n\in \mathbb{N}$ and $g\in D_{2n}\setminus S_{2n}$. We have
$g\Gamma_{2n}\cap v\Gamma_{2n}=\emptyset$, for every $v\in S_{2n}.$
Since $v\Gamma_{2n+1}\subseteq v\Gamma_{2n}$, then 
$$
g\Gamma_{2n}\cap v\Gamma_{2n+1}=\emptyset \mbox{ for every } v\in S_{2n}.
$$
For $m>n$, observe that $S_{2m}\subseteq v_{2n}\Gamma_{2n}\dots \Gamma_{2m-1}\subseteq v_{2n}\Gamma_{2n}.$  Thus
$$
S_{2m}\Gamma_{2m+1}\subseteq v_{2n}\Gamma_{2n}\Gamma_{2m+1}\subseteq v_{2n}\Gamma_{2n}.
$$
Since $g\neq v_{2n}\in D_{2n}$, then $g\Gamma_{2n}\cap v_{2n}\Gamma_{2n}=\emptyset.$ This implies 
$$g\Gamma_{2n}\cap S_{2m}\Gamma_{2m+1}=\emptyset \mbox{ for every } m\geq n.$$
If $g\Gamma_{2n}\cap S_{2m}\Gamma_{2m+1}=\emptyset$ for every $m<n$, then $g\Gamma_{2n}\cap \bigcup_{m\geq 0}S_{2m}\Gamma_{2m+1}=\emptyset$. This implies that $\eta(g\gamma)= 1$, for every $\gamma\in\Gamma_{2n}$ and then $g\in \Per(\eta, \Gamma_{2n}).$ If there exists $m<n$ such that $g\Gamma_{2n}\cap S_{2m}\Gamma_{2m+1}\neq\emptyset$, then since $\Gamma_{2n}\subseteq \Gamma_{2m+1}$, we have that $g\in \Per(\eta, \Gamma_{2n},0)$. This implies that $D_{2n}\setminus S_{2n}\subseteq \Per(\eta, \Gamma_{2n})\cap D_{2n}.$

Thus,  
\begin{align*}
    \lim_{n\to\infty}\dfrac{|D_{2n}\setminus S_{2n}|}{|D_{2n}|}\leq\lim_{n\to\infty}\dfrac{|D_{2n}\cap \Per(\eta,\Gamma_{2n})|}{|D_{2n}|}. 
\end{align*}
Note that $$|v_{2n-1}\Gamma_{2n-1}\cap D_{2n}|=|\Gamma_{2n-1}\cap D_{2n}|=\frac{|D_{2n}|}{|D_{2n-1}|}.$$
Note also that $S_{2n}=(v_{2n-1}\Gamma_{2n-1}\cap D_{2n})\setminus \{v_{2n-1}\}$. Indeed, $v_{2n-1}\Gamma_{2n-1}\cap D_{2n}\cap D_{2n-1}= v_{2n-1}\Gamma_{2n-1}\cap D_{2n-1}$, and since $D_{2n-1}$ is a fundamental domain of $\Gamma_{2n-1}$, the only element in this intersection is $v_{2n-1}$. 

We conclude that
$$
|D_{2n}\setminus S_{2n}|=|D_{2n}|-|S_{2n}|=|D_{2n}|-\left(\frac{|D_{2n}|}{|D_{2n-1}|}-1\right).
$$

Therefore,
\begin{align*}
  \lim_{n\to\infty}\dfrac{|D_{2n}\setminus S_{2n}|}{|D_{2n}|}=1-\lim\limits_{n\to\infty}\dfrac{1}{|D_{2n-1}|}+\lim\limits_{n\to\infty}\dfrac{1}{|D_{2n}|} =1.
\end{align*}
Consequently, $\lim\limits_{n\to\infty}\dfrac{|D_n\cap \Per(\eta,\Gamma_n)|}{|D_n|}=1$, i.e., $\eta$ is a regular Toeplitz array.

\end{proof}

\begin{proof}[Proof of Theorem \ref{theo:main1}]
This is direct from Propositions \ref{regular-measure}, \ref{prop:regularexistence}, \ref{1-1-extension}  and Lemma \ref{odometers-compactification}.
\end{proof}

\section{$G$-subshifts with at least $r$ ergodic measures}\label{sec:finitely-many}

Let $r>1$ be an integer. In this section we  prove that for any residually finite group $G$, there exists a  Toeplitz $G$-subshift having at least $r$ ergodic measures. In the case where $G$ is also amenable, the Toeplitz $G$-subshift has exactly $r$ ergodic measures and we recover a result in \cite{CP14} about the realization of Choquet simplices as sets of invariant measures of subshifts of residually finite amenable groups, for the case of finite dimensional Choquet simplices (see Remark \ref{remark:amenable-2}). 

Through this section, $G$ will be a  residually finite group and $(\Gamma_i)_{i\geq 0}$ a  strictly decreasing sequence  of finite index normal subgroups of $G$ with trivial intersection. Fixing $\Gamma_0=G$ and  taking subsequences if necessary, by Lemma \ref{decom}, we can assume that $(D_i)_{i\geq 0}$ is a sequence such that $D_i$ is a fundamental domain of $G/\Gamma_i$ verifying the following conditions:
\begin{enumerate}
    \item $\{1_G\}\subseteq D_i\subseteq D_{i+1}$
    \item $G=\bigcup_{i=1}^\infty D_i$.
    \item $D_j=\bigcup_{v\in D_j\cap \Gamma_{i}} vD_i$, for each $j>i\geq 1$.
\end{enumerate}
 Furthermore, up to taking subsequences again, we can also assume that    $([\Gamma_i:\Gamma_{i+1}])_{i\geq 0}$ grows as fast as needed. From now on, we will assume that for every $i\geq 0$,
   
 $$[\Gamma_i: \Gamma_{i+1}] > \frac{1}{1-(\frac{1}{2})^{(\frac{1}{2})^{i+1}}},$$
or, equivalently,
\begin{equation}\label{crecimiento}
\frac{|D_i|}{|D_{i+1}|}< 1-\left(\frac{1}{2}  \right)^{(\frac{1}{2})^{i+1}}.
\end{equation}

 \begin{lemma}\label{crecimiento-2}
For every $l\geq 0$ we have
$$
\lim_{n\to \infty}\prod_{j=l}^n \left(1-\frac{|D_j|}{|D_{j+1}|}\right)\geq \left(\frac{1}{2}\right)^{(\frac{1}{2})^l}. 
$$
      
 \end{lemma}
\begin{proof}
Let $a_{l,n}=\prod_{j=l}^n \left(1-\frac{|D_j|}{|D_{j+1}|}\right).$ Condition (\ref{crecimiento}) implies that
$$
\log_2{a_{l,n}}\geq -\left(\frac{1}{2}\right)^{l+1}\sum_{j=0}^{n-l}\left(\frac{1}{2}\right)^{j}. 
$$
From this we get
$\lim_{n\to \infty}\log_2{a_{l,n}}\geq -\left( \frac{1}{2}\right)^l$. Since the function $\log_2$ is continuous and increasing, we deduce
$$
\lim_{n\to \infty}a_{l,n}\geq \left( \frac{1}{2}\right)^{(\frac{1}{2})^l}.
$$
\end{proof}

\subsection{Construction of the Toeplitz array.}\label{subsec:toeplitzconstruction}
Consider the alphabet $\Sigma=\{1,2,\dots,r\}$. Inspired by the ideas in  \cite{Williams}, we will construct a Toeplitz element $\eta\in \Sigma^{G}$ such that $X=\overline{O_{\sigma}(\eta)}$ has at least $r$ ergodic measures.   

Fix the sequence $(\alpha_i)_{i\geq 0}$ in $\Sigma$ given by  $\alpha_i=j$ whenever $j\equiv i \;(\bmod\; r)$. We will define the Toeplitz sequence $\eta$ in several steps:
 
 \medskip
 
 {\bf Step 1:} We define $J(0)=\{1_G\}$ and  $\eta(g)=\alpha_1$, for every $g\in \Gamma_1$.
 
 \medskip 
 
 {\bf Step 2:} We define
    \begin{align*}
        J(1)= D_1\setminus \Gamma_1. 
    \end{align*}
We put $\eta(hg)=\alpha_2$ for every  $h\in J(1)$ and $g\in \Gamma_2$.

Note that $J(1)$ is the set of elements in $D_1$ which were not defined in the previous step. Thus, $\eta$ is defined in the whole set $D_1$ in this second step.
 
   \medskip
   
  {\bf Step m+1: } We define
  $$J(m)=  D_m\setminus \bigcup_{i=0}^{m-1}J(i)\Gamma_{i+1}.$$  Define $\eta(hg)=\alpha_{m+1}$ for every  $h\in J(m)$ and $g\in \Gamma_{m+1}$.
    
 Observe that $J(m)$ is the set of elements in $D_m$ which were not defined in the previous steps. Thus  $\eta$ is defined in the whole set $D_m$ during the step $m+1$.

\medskip 

This construction produces a non-periodic Toeplitz sequence since  $G$ is equal to the disjoint union of the sets  $\{J(i)\Gamma_{i+1}: i\geq 0\}$.

 \medskip

\begin{lemma}\label{auxiliar0}
For every $n\geq 0$, we have
$$
J(n+1)=\bigcup_{\gamma\in (D_{n+1}\cap \Gamma_n)\setminus\{1_G\}}\gamma J(n).
$$
\end{lemma}
\begin{proof}
If $u\in D_{n+1}$, then $u=\gamma v$ for some $\gamma\in D_{n+1}\cap \Gamma_n$ and $v\in D_n$.  Thus if $u\in J(n+1)=D_{n+1}\setminus \bigcup_{l=0}^{n}J(l)\Gamma_{l+1}$, then $\gamma\neq 1_G$. 
Furthermore, if $v\in \bigcup_{l=0}^{n-1} J(l)\Gamma_{l+1}$, then $u\in \gamma\bigcup_{l=0}^{n-1}J(l)\Gamma_{l+1}=\bigcup_{l=0}^{n-1}J(l)\Gamma_{l+1}$, which is impossible. Therefore, we obtain that $v\in J(n)$ and we can conclude $J(n+1)\subseteq \bigcup_{\gamma\in (D_{n+1}\cap \Gamma_n)\setminus\{1_G\}}\gamma J(n)$.
On the other hand, 
  if $v\in J(n)$ then  for $\gamma\in( D_{n+1}\cap \Gamma_n)\setminus\{1_G\}$, we obtain that $u=\gamma v\notin \bigcup_{l=0}^{n-1}J(l)\Gamma_{l+1}$. Note that $J(n)\Gamma_{n+1}\cap D_{n+1}=J(n)$. Hence, if $u\in J(n)\Gamma_{n+1}\cap D_{n+1}=J(n)$, we obtain that $\gamma=1_G$, a contradiction.  From this we get
  $\bigcup_{\gamma\in (D_{n+1}\cap \Gamma_n)\setminus\{1_G\}}\gamma J(n)\subseteq J(n+1)$. 
\end{proof}

\medskip

 \begin{lemma}\label{lemma-Per}
     Let $\eta\in\Sigma^G$ be the Toeplitz sequence constructed before. Then,  
     \begin{align*}
         \Per(\eta, \Gamma_i)=\bigcup_{l=0}^{i-1}J(l)\Gamma_{l+1}, \mbox{ for every } i\in\mathbb{N}.
     \end{align*}
 \end{lemma}
\begin{proof}
It is not difficult to verify that $J(i)\Gamma_{i+1}\cap J(k)\Gamma_{k+1}=\emptyset$, for every $i\neq k$. It is also straightforward to check that
\begin{equation}\label{G-union}
    G=\bigcup_{i\geq 0}J(i)\Gamma_{i+1}.
\end{equation}
Let $k\geq i\geq 0$. We will show that $J(k)\Gamma_{k+1} \subseteq G\setminus \Per(\eta,\Gamma_k)$. For that observe that if 
 $m\in J(k)$, then $\eta(m)=\alpha_{k+1}$. On the other hand,  from Lemma \ref{auxiliar0}, we have that  $\gamma m\in J(k+1)$   for every $\gamma\in (D_{k+1}\cap \Gamma_k)\setminus{1_G}$, which implies that  $\eta(\gamma m)=\alpha_{k+2}$. Since $\alpha_{k+1}\neq \alpha_{k+2}$, we get that $m\notin \Per(\eta, \Gamma_k)$. From this we deduce that $m\Gamma_k\subseteq G\setminus \Per(\eta,\Gamma_k)$, and since $m\in J(k)$ is arbitrary, we conclude that  $J(k)\Gamma_{k+1} \subseteq G\setminus \Per(\eta,\Gamma_k)$. From the relation $\Per(\eta,\Gamma_i)\subseteq \Per(\eta, \Gamma_k)$ follows that 
 $J(k)\Gamma_{k+1} \subseteq G\setminus \Per(\eta,\Gamma_k)\subseteq G\setminus \Per(\eta,\Gamma_i)$. From this we get $\bigcup_{k\geq i}J(k)\Gamma_{k+1}\subseteq G\setminus \Per(\eta,\Gamma_i)$. Then, applying equation (\ref{G-union}), we deduce
 $$
 \Per(\eta,\Gamma_i)\subseteq \bigcup_{k=0}^{i-1}J(k)\Gamma_{k+1}.
 $$
From the construction of $\eta$ we have $\bigcup_{l=0}^{i-1}J(l)\Gamma_{l+1}\subseteq\Per(\eta,\Gamma_i)$,  which implies that both sets are the same.
\end{proof}
\begin{proposition}
$(\Gamma_n )_{n\in \mathbb{N}}$ is a period structure for $\eta$.
\end{proposition}
\begin{proof}
Let $i\geq 1$. 
If $g\in G$ is such that $\Per(\eta, \Gamma_i)\subseteq \Per(\sigma^{g}\eta, \Gamma_i)$, then $g^{-1}w\in \Per(\eta, \Gamma_i)$, for every $w\in \Per(\eta, \Gamma_i)$.  Lemma \ref{lemma-Per} implies that
\begin{equation}\label{period-structure2}
 \bigcup_{l=0}^{i-1}J(l)\Gamma_{l+1}\subseteq \bigcup_{l=0}^{i-1}gJ(l)\Gamma_{l+1}.
\end{equation}
We will prove inductively that $\Gamma_i$ is an essential period of $\eta$.

\medskip

For $i=1$, let $g\in G$ be such that $\Per(\eta,\Gamma_1)\subseteq \Per(\sigma^g\eta,\Gamma_1)$. From equation (\ref{period-structure2}) we get that $\Gamma_1\subseteq g\Gamma_1$, which implies that $g\in \Gamma_1$.

\medskip 

Suppose the result is true for $i\geq 1$. We will prove it for $i+1$. Let $g\in G$ be such that $\Per(\eta, \Gamma_{i+1}, \alpha)\subseteq \Per(\sigma^g\eta,\Gamma_{i+1}, \alpha)$, for every $\alpha\in \Sigma$. Since $\Gamma_{i+1}\subset\Gamma_i$, we have $\Per(\eta,\Gamma_i,\alpha)\subseteq \Per(\eta,\Gamma_{i+1},\alpha)$. Consider $h\in\Per(\eta,\Gamma_i, \alpha)$ and $\gamma_i\in\Gamma_i$.  Let $d\in D_{i+1}$ and $\gamma_{i+1}\in\Gamma_{i+1}$ be such that $\gamma_i=d\gamma_{i+1}$. Then $$h\gamma_i\gamma_{i+1}^{-1}=hd\in\Per(\eta,\Gamma_i,\alpha)\subseteq\Per(\eta,\Gamma_{i+1},\alpha)\subseteq\Per(\sigma^g\eta,\Gamma_{i+1},\alpha).$$
    Thus $\alpha=\sigma^g\eta(hd)= \sigma^g\eta(hd\gamma_{i+1})=\sigma^g\eta(h\gamma_i)$. Since $\gamma_i\in \Gamma_i$ was arbitrarily taken, we deduce that $h\in\Per(\sigma^g\eta,\Gamma_{i}, \alpha)$. Because this is true for every $h\in\Per(\eta,\Gamma_i,\sigma)$, the hypothesis implies $g\in\Gamma_i$.

Hence, we get $J(l)\Gamma_{l+1}=gJ(l)\Gamma_{l+1}$, for every $0\leq l\leq i-1$. By Equation (\ref{period-structure2}) this implies that
$
J(i)\Gamma_{i+1}\subseteq gJ(i)\Gamma_{i+1}.
$
From this we get that for  $u\in J(i)\subseteq D_i$, there exist $v\in J(i)\subseteq D_i$ and $\gamma\in \Gamma_{i+1}$ such that $u=gv\gamma$. 
Since $g\in \Gamma_i$ and $\Gamma_i$ is normal,  there exists $g'\in \Gamma_i$ such that $u=gv\gamma=vg'\gamma\in v\Gamma_i$. Since $u$ and $v$ belong to $D_i$, we deduce that $v=u$ and then $g'=\gamma^{-1}\in \Gamma_{i+1}$. Finally, $gv=vg'\in v\Gamma_{i+1}=\Gamma_{i+1}v$, which implies that $g\in \Gamma_{i+1}$.  
\end{proof}

\subsection{Construction of periodic measures on $\Sigma^G$.} In this section we will define a sequence of periodic invariant measures defined on $\Sigma^G$. Then we will show that the accumulation points of this sequence are supported on $X=\overline{O_{\sigma}(\eta)}$.   Unlike the non-amenable case, when $G$ is amenable it is  guaranteed that these limit measures are supported on $X$ (see Remark  \ref{remark:amenable-1} below).

\medskip

For every $n\geq 1$ we define $\eta_n\in \Sigma^G$ as
$$
\eta_n(\gamma D_n)=\eta(D_n), \mbox{ for every } \gamma\in \Gamma_n,
$$
that is, $\eta_n(\gamma g)=\eta(g)$ for every  $\gamma \in \Gamma_n$ and $g\in D_n$. Thus we have $\sigma^{\gamma}(\eta_n)=\eta_n$, for every $\gamma\in \Gamma_n$, which implies that
$$
O_{\sigma}(\eta_n)=\{\sigma^{u^{-1}}(\eta_n): u\in D_n\}. 
$$ 

We define the following probability $\sigma$-invariant Borel measure on $\Sigma^G$,
$$
\mu_n=\frac{1}{|D_n|}\sum_{u\in D_n}\delta_{\sigma^{u^{-1}}(\eta_n)}.  
$$

Let $i\in \{1,\dots, r\}$ and $[i]$ be the subset of all $x\in \Sigma^G$ such that $x(1_G)=i$.

For every $n$ such that $n+1\equiv i \;(\bmod\; r)$ we have
$$
\mu_n([i])=\frac{|J(n)|+|\Per(\eta,\Gamma_n,i)\cap D_n|}{|D_n|}\geq \frac{|J(n)|}{|D_n|}=1-\frac{|D_n\cap \Per(\eta, \Gamma_n)|}{|D_n|}, 
$$
$$
\mu_n([j])=\frac{|\Per(\eta, \Gamma_n,j)\cap D_n|}{|D_n|}\leq \frac{|D_n\cap \Per(\eta, \Gamma_n)|}{|D_n|}, \mbox{ for } j\in \{1,\dots, r\}\setminus\{i\}.
$$
Recall that in Section \ref{sec:regular} it was shown that $d_n=\frac{|D_n\cap \Per(\eta, \Gamma_n)|}{|D_n|}$ defines an increasing sequence converging to some $d\in [0,1]$. Using the same argument, it is possible to show that $d_{n,j}=\frac{|D_n\cap \Per(\eta, \Gamma_n,j)|}{|D_n|}$ also defines an increasing sequence in $[0,1]$. This implies that every accumulation point $\mu$ of  $(\mu_{i+sr-1})_{ s\geq 0}$ satisfies 
$$
\mu([i])=1-d+\lim_{n\to \infty}\frac{|\Per(\eta, \Gamma_n,i)\cap D_n|}{|D_n|}\geq 1-d,
$$
and
$$
\mu([j])=\lim_{n\to \infty}\frac{|\Per(\eta, \Gamma_n,j)\cap D_n|}{|D_n|}=t_j\leq d \mbox{ for every } j\in \{1,\dots, r\}\setminus\{i\}.
$$

\begin{proposition}\label{regular}
For the Toeplitz array $\eta$ defined above we have   
$$
1-d_{n+1}=\left(1-\frac{1}{|D_{1}|}\right)\prod_{j=1}^n \left(1-\frac{|D_j|}{|D_{j+1}|}\right), \mbox{ for every } n\in \mathbb{N}.
$$
This implies that  $d<1-d$.
\end{proposition}
\begin{proof}
 Since $(\Per(\eta,\Gamma_{n+1})\setminus \Per(\eta,\Gamma_n))\cap D_{n+1}=D_n\setminus D_n\cap\Per(\eta, \Gamma_n)$, we have 
 \begin{eqnarray*}
  d_{n+1}=\frac{|D_{n+1}\cap \Per(\eta,\Gamma_{n+1})|}{|D_{n+1}|} &  = & \frac{|D_n\cap \Per(\eta, \Gamma_n)||D_{n+1}\cap \Gamma_n|+(|D_n|-|D_n\cap \Per(\eta, \Gamma_n)|)}{|D_{n+1}|}\\
             &= & \frac{|D_n\cap \Per(\eta, \Gamma_n)|}{|D_n|}+\frac{|D_n|}{|D_{n+1}|}\left(1-\frac{|D_n\cap \Per(\eta, \Gamma_n)|}{|D_n|}\right)\\
             & = &d_n+\frac{|D_n|}{|D_{n+1}|}(1-d_n).
\end{eqnarray*}

The previous equation also implies that
\begin{eqnarray*}
    1-d_{n+1}=(1-d_n)-\frac{|D_n|}{|D_{n+1}|}(1-d_n)&= &(1-d_n)\left(1-\frac{|D_n|}{|D_{n+1}|}\right)\\
    &= & (1-d_{n-1})\left(1-\frac{|D_{n-1}|}{|D_{n}|}\right)\left(1-\frac{|D_n|}{|D_{n+1}|}\right)\\
    &= & (1-d_1)\prod_{j=1}^n \left(1-\frac{|D_j|}{|D_{j+1}|}\right).
\end{eqnarray*}
Since $d_1=\frac{1}{|D_1|}$ we get the desired equality. Taking limit on $n$ and applying Lemma \ref{crecimiento-2} for $l=1$, we get $1-d\geq (1-d_1)\frac{1}{\sqrt{2}}.$ On the other hand, condition (\ref{crecimiento}) for $i=0$ implies that $1-d_1>\frac{1}{\sqrt{2}}$, from which we get $1-d>\frac{1}{2}$ and then $d<1-d$. 
\end{proof}

\begin{remark}\label{different-measures}{\rm From Proposition \ref{regular}, 
we get that there are at least $r$ different accumulation points of $(\mu_n)_{n\in\mathbb{N}}$, namely $\nu_1,\dots, \nu_r$, where    $\nu_j$ is an accumulation point of $(\mu_{j+sr-1})_{ s\geq 0}$  for every $j\in \{1,\dots, r\}$.  Furthermore, if $\mu$ is an accumulation point of $(\mu_n)_{n\in\mathbb{N}}$ then there exists $i\in \{1,\dots, r\}$ such that 
$$(\mu([1]),\dots, \mu([r]))=(\nu_i([1]), \dots, \nu_i([r]))=(t_1,\dots, t_{i-1},1-d +t_i, t_{i+1},\dots, t_r)=\vec{t}_i, $$
where
$$
t_j=\lim_{n\to \infty}\frac{|\Per(\eta, \Gamma_n,j)\cap D_n|}{|D_n|}  \mbox{ for every } j\in \{1,\dots, r\}.
$$
Taking subsequences of $(\Gamma_n)_{n\in\mathbb{N}}$ we can assume that $(\mu_{j+sr-1})_{ s\geq 0}$ converges to $\nu_j$, for every $1\leq j\leq r$. In other words, we can assume that $\nu_1,\dots, \nu_r$ are the unique accumulation points of $(\mu_n)_{n\in\mathbb{N}}$.
}
  \end{remark}

\begin{remark}\label{remark:amenable-1}
{\rm Suppose that  $G$ is amenable. Then we can assume that the sequence $(D_n)_{n\in\mathbb{N}}$  is F\o lner (see \cite{CP14}).  Let $x$ be any element in $\Sigma^G$ and define the measures $\mu_n$ as before, taking $\eta=x$. The accumulation points of $(\mu_n)_{n\in\mathbb{N}}$  are supported on $\overline{O_{\sigma}(x)}$. Indeed,  if  $U\subseteq \Sigma^G$ is a cylinder set given by fixing the coordinates of its points in some finite set $F\subseteq G$, then
$$
\mu_n(U)=\frac{|\{ v\in \partial_F D_n: \sigma^{v^{-1}}\eta_n\in U\} |}{|D_n|}+\frac{|\{ v\in D_n\setminus \partial_F D_n: \sigma^{v^{-1}}\eta_n\in U\} |}{|D_n|},
$$
where $\partial_F D_n=\{v\in D_n: vF\not\subseteq D_n\}$. If  $(D_n)_{n\in\mathbb{N}}$ is F\o lner then the first term of the sum goes always to zero with $n$. This implies that $\nu_i(U)>0$ for some $1\leq  i \leq r$, only if $U$ intersects the orbit of $x$, from which we deduce that the measure $\nu_1,\dots, \nu_r$ are supported on $\overline{O_{\sigma}(x)}$.     

When the group $G$ is non-amenable, the supports of the accumulation points of $(\mu_n)_{n\in \mathbb{N}}$ depend on the choice of $x$. Indeed, if $x$ is an element of a subshift without invariant measures (subshifts are a test family for amenability, see \cite{FJ21}) then the accumulation points of $(\mu_n)_{n\in\mathbb{N}}$ can not be supported on $\overline{O_{\sigma}(x)}$.
 }
\end{remark}
 
 In the sequel, we will show that the accumulation points of $(\mu_n)_{n\in\mathbb{N}}$ are supported on $X$ (regardless of whether $G$ is amenable or not).

\begin{lemma}\label{good-relation1}
Let $n\geq 1$. For every $m\geq n+2$ there exists 
\begin{equation}\label{good-relation}
\gamma\in (\Gamma_{n+1}\cap D_m)\setminus (D_{n+1}\Gamma_{n+2} \cup  \dots \cup D_{m-1}\Gamma_m).
\end{equation}
Moreover,
$$
\left|(\Gamma_{n+1}\cap D_m)\setminus (D_{n+1}\Gamma_{n+2} \cup  \dots \cup D_{m-1}\Gamma_m)\right|\geq  \frac{|D_{m}|}{|D_{n+1}|}\prod_{l=1}^{m-n-1}\left(1-\frac{|D_{n+l}|}{|D_{n+l+1}|}   \right).
$$

Furthermore, if $\gamma$ satisfies (\ref{good-relation}), then
$$
\gamma D_{n+1}\subseteq  D_m\setminus (D_{n+1}\Gamma_{n+2} \cup  \dots \cup D_{m-1}\Gamma_m).
$$
\end{lemma}
 \begin{proof}
 For $m=n+2$, observe that $\gamma\in (\Gamma_{n+1}\cap D_{n+2})\setminus\{1_{G}\}$ satisfies the property. Indeed, if $\gamma=u\gamma'$ for $u\in D_{n+1}$ and $\gamma'\in \Gamma_{n+2}$, then $u=1_G$. This implies $\gamma=\gamma'\in \Gamma_{n+2}$, but since $\gamma\in D_{n+2}$, this is only possible if $\gamma=1_G$. 
 
We will continue by induction on $m\geq n+2$.  Suppose  there exists $$\gamma\in (\Gamma_{n+1}\cap D_m)\setminus (D_{n+1}\Gamma_{n+2} \cup  \dots \cup D_{m-1}\Gamma_m).$$
 
 Let $\gamma_0\in (\Gamma_m\cap D_{m+1})\setminus \{1_G\}$. Since $\gamma\in \Gamma_{n+1}\cap D_m$ and 
 $$
 D_{m+1}=\bigcup_{\gamma'\in \Gamma_m\cap D_{m+1}}\gamma'D_m,
 $$
 we have that $\gamma_0\gamma\in \Gamma_{n+1}\cap D_{m+1}.$  Suppose there exist $n+1\leq s\leq m-1$, $u\in D_s$ and $\gamma'\in \Gamma_{s+1}$ such that $\gamma_0\gamma=u\gamma'$. Since $s+1\leq m$ and $\gamma_0\in \Gamma_m$, this implies that $\gamma\in D_s\Gamma_{s+1}$ which is a contradiction with the choice of $\gamma$. On the other hand,  since $\gamma_0\gamma \in D_{m+1}$ the only way that  $\gamma_0\gamma\in D_m\Gamma_{m+1}$ is having  $\gamma_0\gamma\in D_m$, which is only possible if $\gamma_0=1_G$ (because $\gamma\in D_m$ and $\gamma_0\in \Gamma_m$). We have shown that
 $$\gamma_0\gamma\in (\Gamma_{n+1}\cap D_{m+1})\setminus (D_{n+1}\Gamma_{n+2} \cup  \dots \cup D_{m}\Gamma_{m+1}).$$
 This implies that
$$
N_{m+1,n}\geq \left(\frac{|D_{m+1}|}{|D_m|}-1\right)N_{m,n}, 
$$
where
$$
N_{m+1,n}=|(\Gamma_{n+1}\cap D_{m+1})\setminus (D_{n+1}\Gamma_{n+2} \cup  \dots \cup D_{m}\Gamma_{m+1})|.
$$
From this we get
$$
N_{m+1,n}\geq \prod_{l=1}^{m-n}\left(\frac{|D_{n+l+1}|}{|D_{n+l}|}-1\right)=\frac{|D_{m+1}|}{|D_{n+1}|}\prod_{l=1}^{m-n}\left(1-\frac{|D_{n+l}|}{|D_{n+l+1}|}   \right).
$$

\medskip

Finally, let $\gamma$ be an element  in $\Gamma_{n+1}$ satisfying relation (\ref{good-relation}) and let $u\in D_{n+1}$.  Suppose there exist $n+1\leq s\leq m-1$, $v\in D_s$ and $\gamma'\in \Gamma_{s+1}$ such that $\gamma u=v\gamma'$. We can write $v=\gamma'' u'$, with $u'\in D_{n+1}$ and $\gamma''\in \Gamma_{n+1}\cap D_s$.  The equation $\gamma u=\gamma'' u' \gamma'$ implies $u=u'$ and then $\gamma\in D_s\Gamma_{s+1}$, which is not possible. This finishes the proof.
  \end{proof}
 
 Let $n\geq 1$. We define  
 \begin{equation}\label{definition}
 U_n=\{x\in \Sigma^G: x(D_{n+1})=\eta_n(D_{n+1})\}.
 \end{equation}
Observe that $U_n$ is the set of all $x\in \Sigma^G$ such that $x(\gamma D_n)=\eta(D_n)$, for every $\gamma\in D_{n+1}\cap \Gamma_n$.

 \begin{lemma}\label{good-patches}
Let $n\geq 1$ and  $m>n$ be such that $m \equiv n\; (\bmod\; r)$. If $\gamma \in \Gamma_{n+1}\cap D_m$ satisfies relation (\ref{good-relation}), then $\sigma^{\gamma^{-1}}(\eta)\in U_n$. This implies that  $U_n\cap O_{\sigma}(\eta)\neq \emptyset.$  
\end{lemma}

\begin{proof}

Let $m>n$ be such that $m \equiv n \;(\bmod\; r)$ and  let $\gamma_0\in \Gamma_{n+1}\cap D_m$ an element of the group satisfying the relation (\ref{good-relation}).  From the choice of $\gamma_0$, if $g\in D_{n+1}$  then $\eta(\gamma_0g)$ has been defined in step $k\in \{1,\dots, m\}$ if and only if $g\in \Per(\eta, \Gamma_{n+1})$. Observe this implies
 \begin{equation}\label{eq2}
 \eta(\gamma_0D_n)=\eta(D_n)
 \end{equation}
    and
 \begin{equation}\label{eq3}
 \eta(\gamma_0\gamma u)=\eta(u), \forall  u\in D_n\cap \Per(\eta, \Gamma_n) \mbox{ and } \gamma\in D_{n+1}\cap \Gamma_n.
 \end{equation}
  If $u\in D_n\setminus \Per(\eta, \Gamma_n)=J(n)$ then
 $\gamma_0\gamma u\in J(m)$, which implies that
 \begin{equation}\label{eq4}
 \eta(\gamma_0\gamma u)=\alpha_{m+1}=\alpha_{n+1}=\eta(u).
 \end{equation}
 
 From (\ref{eq2}), (\ref{eq3}) and (\ref{eq4}) we get
 $$
 \eta(\gamma_0\gamma D_n)=\eta(D_n), \mbox{ for every } \gamma\in D_{n+1}\cap \Gamma_n, 
 $$ 
 which implies that
 $$
 \eta(\gamma_0D_{n+1})=\eta_n(D_{n+1}),
 $$
 and then $\sigma^{\gamma_0^{-1}}(\eta)\in U_n$. 
\end{proof}

\begin{proposition}\label{several-measures}
 The accumulation points of $(\mu_n)_{n\in\mathbb{N}}$ are supported on $X=\overline{O_{\sigma}(\eta)}$. Therefore, $X$ has at least $r$ different invariant probability measures. 
 \end{proposition}
\begin{proof}  
Let $\nu$ be an accumulation point of $(\mu_n)_{n\in\mathbb{N}}$. We have that $\nu=\nu_i$  for some $i\in \{1,\dots, r\}$.

Let $C\subseteq \Sigma^G$ be a clopen set such that $\nu(C)=\varepsilon>0$. We can assume that $C=\{y\in \Sigma^G: y(D_n)=P\}$, where $P$ is some element in $\Sigma^{D_n}$, for some fixed $n\geq 1$.  Since $\nu(C)>0$, we have that $\mu_{i+rs-1}(C)>0$, for infinitely many $s$'s. This implies that $O_{\sigma}(\eta_m)\cap C\neq \emptyset$  for infinitely many $m$'s  which are equal $\bmod\; r$. From this we get  there exists $u_m\in D_m$ such that $\eta_m(u_mD_n)=P$. We can always assume that  $D_m\cdot D_m \subseteq D_{m+1}$, which implies that $u_mD_m\subseteq D_{m+1}$.
From Lemma \ref{good-patches}, it follows there exists $g\in G$ such that $\sigma^g(\eta)\in C$. Therefore $\nu$ is supported on $\overline{O_{\sigma}(\eta)}$.

From Remark \ref{different-measures} we deduce that $X$ has at least $r$ different invariant probability measures. 
\end{proof}

\subsection{Lower bound for the number of ergodic measures.} In this section we   show that   $X$ has at least $r$ ergodic measures.  First we need the following Lemma.

\begin{lemma}\label{auxiliar}
For every $i\geq 1$ and $\gamma\in \Gamma_i$, there exists $l\geq i$ such that $\gamma J(i)\subseteq J(l)\Gamma_{l+1}$. 
\end{lemma}
\begin{proof}
Since $J(i)=D_i\setminus \Per(\eta, \Gamma_i)$, then $\gamma J(i)\cap \Per(\eta, \Gamma_i)=\emptyset$. This implies that
$$
\gamma J(i)\subseteq \bigcup_{l\geq i}J(l)\Gamma_{l+1}.
$$
Let $l=\min\{k\geq i: \gamma J(i)\cap J(k)\Gamma_{k+1}\neq \emptyset\}$.  Let $u\in J(i)$ be such that $\gamma u= v_l \gamma_{l+1}$, for some $v_l\in J(l)$ and $\gamma_{l+1}\in \Gamma_{l+1}$. Since $v_l\in D_l$, there exist $v\in D_i$ and $\gamma'\in \Gamma_i\cap D_l$ such that $v_l=\gamma'v$. The relation $\gamma u= v_l \gamma_{l+1}$ implies $v=u$ and $\gamma=\gamma'\gamma_{l+1}'$, for some $\gamma_{l+1}'\in \Gamma_{l+1}$.  Thus if $s\in J(i)$ then $\gamma s= \gamma'\gamma_{l+1}'s=\gamma's\gamma_{l+1}''$, for some $\gamma_{l+1}''\in \Gamma_{l+1}$. This implies that $\gamma s\in D_l\Gamma_{l+1}\subseteq J(l)\Gamma_{l+1} \cup \bigcup_{k=0}^{l-1}J(k)\Gamma_{k+1}$.  The choice of $l$ implies that $\gamma s\in J(l)\Gamma_{l+1}$,  and then $\gamma J(i)\subseteq J(l)\Gamma_{l+1}$.
\end{proof}

\begin{corollary}\label{partition}
For every $i\geq 0$ and $\gamma\in \Gamma_i$, there exists $\alpha\in \Sigma$ such that 
$$
\eta(g)=\alpha \mbox{ for every } g\in \gamma J(i).
$$
\end{corollary}
\begin{proof}
The case $i=0$ is trivial. Suppose that $i\geq 1$ and $\gamma\in \Gamma_i$.  

From Lemma \ref{auxiliar}, there exists $l\geq i$ such that $\gamma J(i)\subseteq J(l)\Gamma_{l+1}$. By the definition of $\eta$ we get $\eta(g)=\alpha_{l+1}$, for every $g\in \gamma J(i)$.   
\end{proof}

\subsubsection{Partitions and invariant measures}\label{subsec:invmeasuresfinite}
 Recall that $\{\sigma^{v^{-1}}C_n: v\in D_n\}$ is a clopen partition of $X$, where  
    $$
C_n=\{ x\in X: \Per(x,\Gamma_n, \alpha)=\Per(\eta, \Gamma_n, \alpha), \mbox{ for every } \alpha\in \Sigma\}.
$$
  
For every $1\leq i\leq r$, let 
  $$C_{n,i}=\{x\in C_n: x(g)=i \mbox{ for every } g\in J(n)\}.$$
Corollary \ref{partition} implies $\{C_{n,i}: 1\leq i\leq r\}$ is a covering of $C_n$, therefore, $$\cP_n=\{\sigma^{v^{-1}}C_{n,i}: 1\leq i\leq r, v\in D_n\}$$ is a clopen partition of $X$.

  \begin{lemma}\label{rel-partition}
 For every $n\geq 1$ and $1\leq j\leq r$, we have
  \begin{enumerate}
  \item $C_{n+1}\subseteq C_{n,\alpha_{n+1}}$, and
  \item $\sigma^{\gamma^{-1}}C_{n+1,j}\subseteq C_{n,j}$, for every $\gamma\in (\Gamma_{n}\cap D_{n+1})\setminus \{1_G\}$. 
  \end{enumerate}
  \end{lemma}
  \begin{proof}
  Since $\Per(\eta, \Gamma_{n})\subseteq \Per(\eta, \Gamma_{n+1})$, we have $C_{n+1}\subseteq C_n$. Furthermore, $D_n\subseteq \Per(\eta, \Gamma_{n+1})$, which implies that
  $x(D_n)=\eta(D_n)$, for every $x\in C_{n+1}$. In particular, $x(g)=\eta(g)=\alpha_{n+1}$, for every $g\in J(n)$. From this we get $C_{n+1}\subseteq C_{n,\alpha_{n+1}}$.  
  
  \medskip
  
Using that $C_{n+1,j}\subseteq C_n$, we get that for every $\gamma\in \Gamma_{n}$,  $\sigma^{\gamma^{-1}}C_{n+1,j}\subseteq C_n$. On the other hand, if $\gamma\in (\Gamma_n\cap D_{n+1})\setminus \{1_G\}$ and $y\in C_{n+1,j}$, then Lemma \ref{auxiliar0} implies that $\sigma^{\gamma^{-1}}(y)(g)=y(\gamma g)=j$, for every $g\in J(n)$. This shows that $\sigma^{\gamma^{-1}}C_{n+1,j}\subseteq C_{n,j}$. 
  \end{proof}

 Let $\triangle$ be the convex generated by the vectors $\{\vec{t}_1, \dots, \vec{t}_r\}\in \mathbb{R}^r$ (see definition in Remark \ref{different-measures}). That is,
 $$
 \triangle=\left\{\sum_{i=1}^r\alpha_i\vec{t}_i: \sum_{i=1}^r\alpha_i=1, \alpha_1,\dots, \alpha_r\geq 0\right\}.
 $$
 Since the vectors $\vec{t}_1,\dots, \vec{t}_r$ are linearly independent, the convex $\triangle$ is a simplex.

\begin{proposition}\label{at-least}
There is an affine surjective map $p$ from the space of invariant probability measures of $X$ to $\triangle$. Furthermore,  $p(\nu_i)=\vec{t}_i$ for every $1\leq i\leq r$. This implies that $X$ has at least $r$ ergodic measures. 
\end{proposition}

\begin{proof}
 For every $n\in\mathbb{N}$ and $i\in \{1,\dots,r\}$ we set $a_{n,i} =|D_n\cap \Per(\eta, \Gamma_n, i)|$.
 
Let $C_{0,i}=[i]\cap X$. We have that 
    \begin{align*}
        C_{0,i}=\bigcup_{g\in J(n)}\sigma^{g^{-1}} C_{n,i}\cup \bigcup_{g\in \Per(\eta,\Gamma_n,i)\cap D_n}\sigma^{g^{-1}}C_n.
    \end{align*}
    Therefore, for every invariant probability measure $\mu$ of $X$ we have
    \begin{eqnarray*}
        \mu(C_{0,i}) & = & |J(n)|\mu(C_{n,i})+a_{n,i} \mu(C_n)\\
          &=& (|J(n)|+a_{n,i})\mu(C_{n,i})+a_{n,i}\sum_{j\neq i}\mu(C_{n,j})\\
           &=& \frac{|J(n)|+a_{n,i}}{|D_n|}\mu(C_{n,i})|D_n|+\frac{a_{n,i}}{|D_n|}\sum_{j\neq i}\mu(C_{n,j})|D_n|.\\
    \end{eqnarray*}
    Taking a subsequence $(n_k)_k$ in order that $\lim_{k\to \infty}\mu(C_{n_k,j})|D_{n_k}|=\alpha_j\in [0,1]$, we get that
    $$
    \mu(C_{0,i})=\sum_{j=1}^r\vec{t}_j(i)\alpha_j,
    $$
    with $\sum_j\alpha_j=1.$ This implies that $(\mu(C_{0,1}),\dots, \mu(C_{0,r}))$ belongs to $\triangle$. The map $\mu\mapsto (\mu(C_{0,1}),\dots, \mu(C_{0,r}))$ is an affine map from the set of invariant probabilty measures of $X$ to $\triangle$, sending an  accumulation  point $\nu_i$  of $(\mu_{i+kr-1})_{k\in\mathbb{N}}$ to $\vec{t}_i$. We will call this map as $p$. Furthermore, we have that $p$ is surjective and then the set of invariant probability measures of $X$ has at least $r$ extreme points (observe that if $n<r$, then the linear maps from a  $n$-dimensional vector space  to a $r$-dimensional one, cannot be surjective). 

\end{proof}

\begin{lemma}\label{partitions-determine-measures}
Let $p$ be the surjective affine map from the space of invariant probability measures of $X$ to $\triangle$ introduced in Proposition \ref{at-least}. If $\mu$ and $\nu$ are two invariant measures of $X$ such that $p(\mu)=p(\nu)$, then $\mu|_{\mathcal{P}_n}= \nu|_{\mathcal{P}_n}$, for every $n\in \mathbb{N}$.
\end{lemma}
  
  \begin{proof}
For every $n\geq 0$, let $A_n$ be the $r$-dimensional integer matrix given by
  $$
  A_n(i,j)=\left\{ \begin{array}{ll}
                   \frac{|D_{n+1}|}{|D_n|}-1 &\mbox{ if } i=j\neq \alpha_{n+1}\\
                     \frac{|D_{n+1}|}{|D_n|} &\mbox{ if } i=j= \alpha_{n+1}\\
                   0 & \mbox{ if } i\neq j \mbox{ and } i\neq \alpha_{n+1}\\
                   1 & \mbox{ if } i\neq j \mbox{ and } i=\alpha_{n+1}.
     \end{array}\right.,
  $$
 where $|D_0|=1$. 
 
  From Lemma \ref{rel-partition},  for every $1\leq i\leq r$  we have
   $$
  C_{n,i}=\left\{\begin{array}{ll}
                  \bigcup_{\gamma\in (D_{n+1}\cap \Gamma_n)\setminus \{1_G\}}\sigma^{\gamma^{-1}}C_{n+1,i} & \mbox{ if } i\neq \alpha_{n+1}\\
                  \bigcup_{\gamma\in (D_{n+1}\cap \Gamma_n)\setminus \{1_G\}}\sigma^{\gamma^{-1}}C_{n+1,\alpha_{n+1}}  \cup \bigcup_{j=1}^rC_{n+1,j} & \mbox{ if } i=\alpha_{n+1}\\
                  \end{array}\right.
  $$
Thus if $\mu$ is an invariant probability measure of $X$, then
  $$
  \mu(C_{n,i})=\left\{\begin{array}{ll}
                  \left(\frac{|D_{n+1}|}{|D_n|}-1\right)  \mu(C_{n+1,i}) & \mbox{ if } i\neq \alpha_{n+1}\\
                  \left(\frac{|D_{n+1}|}{|D_n|}-1\right)  \mu(C_{n+1,\alpha_{n+1}})  + \sum_{j=1}^r\mu(C_{n+1,j}) & \mbox{ if } i=\alpha_{n+1}\\
                  \end{array}\right.
  $$
  In other words, we have $A_n\mu^{(n+1)}=\mu^{(n)}$, where $\mu^{(n)}=(\mu(C_{n,1}),\dots, \mu(C_{n,r}))$.   Since the matrices $A_n$ are invertible (the columns are linearly independent), we have
  $$ \mu^{(n+1)}=A_n^{-1}\dots A_0^{-1}\mu^{(0)}.$$
Using that $\mu^{(0)}=p(\mu)$, we conclude.
  \end{proof}
  
\begin{remark}\label{remark:amenable-2}
{\rm If $G$ is amenable, then the sequence $(D_n)_{n\in\mathbb{N}}$ can be chosen F\o lner. In this case, Lemma  \ref{partitions-determine-measures} implies immediately that the affine map $p$ is a bijection, because the set of points that $(\mathcal{P}_n)_{n\in\mathbb{N}}$ do not separate, has zero measure with respect to any invariant measure (See \cite[Lemma 17]{CeCo19}). When the sequence $(D_n)_{n\in\mathbb{N}}$ is not F\o lner, that set could be {\em a priori} a  full measure set.  }
\end{remark}
 
\section{Measure-theoretic conjugacy.}\label{sec:measure-conjugate}

In this section, we study the properties of the measures $\nu_i$ constructed in Section \ref{sec:finitely-many}. Recall that the map $p:\cM(X,\sigma)\to \triangle$ introduced in Propostion \ref{at-least} is affine and surjective. We start by proving some technical lemmas that allow us to conclude that for any invariant measure $\mu\in\mathcal{M}(X,\sigma)$ which has the same $p$-image of $\nu_i$ for some $1\leq i\leq r$, the p.m.p. dynamical system $(X,\sigma,\mu)$ is measure conjugate to the odometer with its unique invariant measure (Lemma \ref{fundamental}, Proposition \ref{conjugacy}). As a consequence of that, we obtain that any invariant measure  whose image under $p$ is equal to the image of $\nu_i$, coincides with $\nu_i$ (Corollary \ref{cor:injectivep}) and that the $\nu_i's$ are ergodic and the unique measures that maximize the measures of symbol cylinders (Theorem \ref{theo:main2}).\\  
In the rest of this section, $U_n$ is the set defined in (\ref{definition}), i.e,  
 $$ 
 U_n=\{x\in \Sigma^G: x(D_{n+1})=\eta_n(D_{n+1})\}.
$$  
Additionally, for every $1\leq i\leq r$, we fix $\nu_i$ the limit of $(\mu_{i+kr-1})_{k\in \mathbb{N}}$ (see Remark \ref{different-measures}).

\begin{lemma}\label{big-measure}
Let $1\leq i\leq r$. For every $n\geq 1$  such that $n+1\equiv i \;(\bmod\; r)$, we have 
$$
\nu_i(U_n)\geq \lim_{s\to\infty}\frac{1}{|D_{n+1}|}\prod_{l=1}^{sr-1}\left(1-\frac{|D_{n+l}|}{|D_{n+l+1}|}   \right)
$$ 
and
$$
\nu_i\left(\bigcup_{v\in D_{n+1}}\sigma^{v^{-1}}U_n\right)\geq \lim_{s\to\infty} \prod_{l=1}^{sr-1}\left(1-\frac{|D_{n+l}|}{|D_{n+l+1}|}   \right).
$$ 
\end{lemma}
\begin{proof}
Let $k\geq 2$.  From Lemma \ref{good-relation1} we have 
$$
N_{n+k,n}\geq \prod_{l=1}^{k-1}\left(\frac{|D_{n+l+1}|}{|D_{n+l}|}-1\right)=\frac{|D_{n+k}|}{|D_{n+1}|}\prod_{l=1}^{k-1}\left(1-\frac{|D_{n+l}|}{|D_{n+l+1}|}   \right).
$$

On the other hand, if $k=sr$ for some $s\geq 1$, then  Lemma \ref{good-patches} implies that for every $\gamma\in (\Gamma_{n+1}\cap D_{n+sr})\setminus (D_{n+1}\Gamma_{n+2} \cup  \dots \cup D_{n+sr-1}\Gamma_{n+sr})$ we have
$$
\eta_{n}(D_{n+1})=\eta(\gamma D_{n+1})=\eta_{n+sr}(\gamma D_{n+1}).
$$
Thus we have $\sigma^{\gamma^{-1}}\eta_{n+sr}\in U_n$ and $\sigma^{(\gamma v)^{-1}}\eta_{n+sr}\in \sigma^{v^{-1}}U_n$ for $v\in D_{n+1}$, which implies 
$$
\mu_{n+sr}(U_n)\geq N_{n+sr,n}\frac{1}{|D_{n+sr}|}\geq \frac{1}{|D_{n+1}|}\prod_{l=1}^{sr-1}\left(1-\frac{|D_{n+l}|}{|D_{n+l+1}|}   \right),
$$
and
$$
\mu_{n+sr}\left(\bigcup_{v\in D_{n+1}}\sigma^{v^{-1}}U_n\right)\geq N_{n+sr,n}\frac{|D_{n+1}|}{|D_{n+sr}|}\geq \prod_{l=1}^{sr-1}\left(1-\frac{|D_{n+l}|}{|D_{n+l+1}|}   \right).
$$
Since $\nu_i$ is the limit of  $(\mu_{n+sr})_{s\in\mathbb{N}}$ and  $\lim_{s\to\infty} \prod_{l=1}^{sr-1}\left(1-\frac{|D_{n+l}|}{|D_{n+l+1}|}\right)$ exists, we conclude.
\end{proof}
For every $1\leq i\leq r$ and $k\geq 1$, we denote
$$
Y_{i,k}=\bigcap_{\gamma\in \Gamma_{i+kr-1}\cap D_{i+kr}}\sigma^{\gamma}C_{i+kr-1,i}.
$$

\begin{lemma}\label{big-measure1}
Let $1\leq i\leq r$ and $k\geq 1$. Then 
$$U_{i+kr-1}\cap \overline{O_{\sigma}(\eta)}=Y_{i,k}.$$ 
\end{lemma}
\begin{proof}

If $x\in Y_{i,k}$ then $x(\gamma a)=\eta(a)$ and $x(\gamma b)=i$, for every $a\in \Per(\eta, \Gamma_{i+kr-1})$,   $\gamma\in \Gamma_{i+kr-1}\cap D_{i+kr}$ and $b\in J(i+kr-1)$. This implies that\\ $\sigma^{\gamma^{-1}}x(D_{i+kr-1})=\eta(D_{i+kr-1})$ for every $\gamma \in \Gamma_{i+kr-1}\cap D_{i+kr}$, which means that $Y_{i,k}\subseteq U_{i+kr-1}\cap \overline{O_{\sigma}(\eta)}$.

Let $x\in U_{i+kr-1}\cap \overline{O_{\sigma}(\eta)}$. There exist $v\in D_{i+kr}$ and $1\leq j \leq r$ such that $x\in \sigma^{v^{-1}}C_{i+kr,j}$. Let $y\in C_{i+kr,j}$ be such that $x=\sigma^{v^{-1}}y$.  

Since $x\in U_{i+kr-1}$, we have
\begin{equation}\label{rel-1}
 x(\gamma D_{i+kr-1})=\eta(D_{i+kr-1}), \mbox{ for every } \gamma\in \Gamma_{i+kr-1}\cap D_{i+kr}. 
\end{equation}
Since  $y\in C_{i+kr,j}\subseteq C_{i+kr}$, we have $y(g)=\eta(g)$ for every $g\in \Per(\eta, \Gamma_{i+kr})$. In particular,
\begin{equation}\label{rel-2}
y(D_{i+kr-1})=\eta(D_{i+kr-1}).
\end{equation}
Finally, the relation $x=\sigma^{v^{-1}}y$ implies
\begin{equation}\label{rel-3}
y(vD_{i+kr-1})=x(D_{i+kr-1})=\eta(D_{i+kr-1}).
\end{equation}
We will show that $v\in \Gamma_{i+kr-1}$. Suppose that $v\notin \Gamma_{i+kr-1}$ and let $0\leq n< i+kr-1$ be the biggest $n$ verifying $v\in \Gamma_{n}$ (here we assume $\Gamma_0=G$ and $D_0=\{1_G\}$).  
Let $\gamma\in \Gamma_{n+1}\cap D_{i+kr}$ and $u\in D_{n+1}$ be such that $v=\gamma u$.  Since $n+1\leq i+kr-1$, from Equation (\ref{rel-3}) we have
\begin{equation}\label{rel-5}
y(vD_{n+1})=y(\gamma uD_{n+1})=\eta(D_{n+1}).
\end{equation}
Since $G/\Gamma_{n+1}$ is a group, we know that there exists $w\in D_{n+1}$ such that $uw\in \Gamma_{n+1}$. Since $v,\gamma \in \Gamma_n$ and $v=\gamma u$, $u\in \Gamma_n$. Thus, $w\in \Gamma_n\cap D_{n+1}$, which implies $wD_n\subseteq D_{n+1}$.
Then, from Equation (\ref{rel-5}) we get
$$
y(vwD_n)=\eta(wD_n).
$$
Equations (\ref{rel-5}) and (\ref{rel-2}), together with the fact that $\gamma u w \in \Gamma_{n+1}$ and $D_n\subseteq \Per(y, \Gamma_{n+1})$, imply
$$
y(vwD_n)=y(\gamma uwD_n)=y(D_n)=\eta(D_n).
$$
From the last two equations we deduce that $\eta(D_n)=\eta(w D_n)$. However, since $w\in \Gamma_{n}\cap D_{n+1}$, the definition of $\eta$ requires that $w=1_G$. Indeed, if $w\neq 1_G$ then $\eta(wg)=\alpha_{n+2}\neq \eta(g)=\alpha_{n+1}$, for every $g\in J(n)$. But if $w=1_G$ then  $v\in \Gamma_{n+1}$, which contradicts the choice of $n$.  This shows that $v\in \Gamma_{i+kr-1}$.

\medskip

Since $y(g)=\eta(g)$ for every $g\in \Per(\eta, \Gamma_{i+kr})\supseteq \Per(\eta, \Gamma_{i+kr-1})$, we have that $x(g)=y(vg)=y(g)=\eta(g)$ for every $g\in \Per(\eta, \Gamma_{i+kr-1})$. This implies that $x\in C_{i+kr-1}$, which means that for every $\gamma\in \Gamma_{i+kr-1}$ there exists $1\leq j\leq r$ such that $\sigma^{\gamma^{-1}}x\in C_{i+kr-1,j}$. From Equation (\ref{rel-1}) we get that for every $\gamma\in \Gamma_{i+kr-1}\cap D_{i+kr}$ the index $j$ is equal to $i$. This shows that $x\in Y_{i,k}$.

\end{proof}

For every $1\leq i\leq r$ and $k\geq 0$, define
\begin{align}\label{Z}
Z_{i,k}=\bigcup_{v\in D_{i+kr}}\sigma^{v^{-1}}C_{i+kr,i}.
\end{align}

\begin{lemma}\label{big-measure2}
Let $1\leq i\leq r$ and $k\geq 1$. Then 
$$
\bigcup_{v\in D_{i+kr}}\sigma^{v^{-1}}Y_{i,k} \subseteq Z_{i,k} \cup  \bigcup_{j\neq i}\bigcup_{v\in D_{i+kr-1}}\sigma^{v^{-1}}C_{i+kr,j}.
$$
\end{lemma}
\begin{proof}
The inclusion $\bigcup_{v\in D_{i+kr}}\sigma^{v^{-1}}C_{i+kr,i}\subseteq \bigcup_{v\in D_{i+kr}}\sigma^{v^{-1}}Y_{i,k}$ is direct from Lemma \ref{rel-partition}.

Suppose that $x\in \bigcup_{v\in D_{i+kr}}\sigma^{v^{-1}}Y_{i,k}$. Let $v\in D_{i+kr}$,  $\gamma\in D_{i+kr}\cap \Gamma_{i+kr-1}$ and $u\in D_{i+kr-1}$  be such that   $\sigma^{v}(x)\in Y_{i,k}$ and $v=\gamma u$. Then $x\in \sigma^{u^{-1}}C_{i+kr-1,i}$. On the other hand, if $w\in D_{i+kr}$, $\gamma'\in \Gamma_{i+kr-1}\cap D_{i+kr}$,  $u'\in D_{i+kr-1}$ and $1\leq j\leq r$ are such that $x\in \sigma^{w^{-1}}C_{i+kr,j}$ and $w=\gamma' u'$,  then Lemma \ref{rel-partition} implies that $x\in \sigma^{u'^{-1}}C_{i+kr-1,j}$ if $\gamma'\neq 1_G$. From this we deduce that for every $j\neq i$, 
$$
\left( \bigcup_{v\in D_{i+kr}}\sigma^{v^{-1}}Y_{i,k} \right) \cap \left( \bigcup_{v\in D_{i+kr}\setminus D_{i+kr-1}}\sigma^{v^{-1}}C_{i+kr,j}\right) =\emptyset.   $$
This shows the result.

\end{proof}

\begin{lemma}\label{full-measure}
For every $1\leq i\leq r$, we have
$$
\lim_{k\to \infty}\nu_i\left(Z_{i,k}\right)=1.
$$
\end{lemma}
\begin{proof}

We have
\begin{eqnarray*}
\nu_i\left( Z_{i,k} \right)+\frac{|D_{i+kr-1}|}{|D_{i+kr}|} & \geq & \nu_i\left( Z_{i,k} \right)+\nu_i\left( \bigcup_{j\neq i}\bigcup_{w\in D_{i+kr-1}}\sigma^{w^{-1}}C_{i+kr,j} \right)\\
    &=& \nu_i\left(   Z_{i,k} \cup  \bigcup_{j\neq i}\bigcup_{w\in D_{i+kr-1}}\sigma^{w^{-1}}C_{i+kr,j}      \right)
\end{eqnarray*}

The previous equation together with Lemmas \ref{big-measure}, \ref{big-measure1} and \ref{big-measure2} imply
$$
\nu_i\left( Z_{i,k} \right)+\frac{|D_{i+kr-1}|}{|D_{i+kr}|}\geq \nu_i\left(\bigcup_{v\in D_{i+kr}}\sigma^{v^{-1}}Y_{i,k}  \right) \geq \lim_{s\to\infty} \prod_{l=1}^{sr-1}\left(1-\frac{|D_{i+kr-1+l}|}{|D_{i+kr+l}|}   \right).
$$
Then, we have
$$
1+\frac{|D_{i+kr-1}|}{|D_{i+kr}|}\geq \nu_i\left( Z_{i,k} \right)+\frac{|D_{i+kr-1}|}{|D_{i+kr}|}\geq \lim_{s\to\infty} \prod_{l=1}^{sr-1}\left(1-\frac{|D_{i+kr-1+l}|}{|D_{i+kr+l}|}   \right).
$$
 Thus, from condition (\ref{crecimiento}) and Lemma \ref{crecimiento-2}  we get
$$
\lim_{k\to \infty}\nu_i(Z_{i,k})=1.
$$
\end{proof}

\vspace{5mm}

\begin{lemma}\label{intersection-lemma}
For every $1\leq i\leq r$ and $k\geq 0$,  we have
$$ Z_{i,k}\subseteq Z_{i,k+1} \cup \bigcup_{v\in D_{i+(k+1)r-1}}\sigma^{v^{-1}}C_{i+(k+1)r},$$ 
and
\begin{equation}\label{intersection}
Z_{i,k}\subseteq \left( \bigcap_{l\geq k} Z_{i,l} \right)  \cup \bigcup_{l\geq k+1}\bigcup_{v\in D_{i+lr-1}}\sigma^{v^{-1}}C_{i+lr}.
\end{equation}

  \end{lemma}
  \begin{proof}
  Lemma \ref{rel-partition} implies that
  $$
Z_{i,k}\subseteq \bigcup_{v\in D_{i+kr+s}}\sigma^{v^{-1}}C_{i+kr+s,i} \mbox{ for every } 1\leq s<r.  
  $$
 Lemma \ref{rel-partition} also implies that
 $$
 \bigcup_{v\in D_{i+kr+r-1}}\sigma^{v^{-1}}C_{i+kr+r-1,i}\subseteq Z_{i,k+1}\cup \bigcup_{j\neq i}\bigcup_{v\in D_{i+(k+1)r-1}}\sigma^{v^{-1}}C_{i+(k+1)r,j}.
 $$ 
  Combining the two previous equations we get 
  $$ Z_{i,k}\subseteq Z_{i,k+1} \cup \bigcup_{v\in D_{i+(k+1)r-1}}\sigma^{v^{-1}}C_{i+(k+1)r},$$ 
  which implies that
  \begin{eqnarray*}
  Z_{i,k} & = &Z_{i,k}\cap\left(Z_{i,k+1} \cup \bigcup_{v\in D_{i+(k+1)r-1}}\sigma^{v^{-1}}C_{i+(k+1)r}\right)\\
               &\subseteq &( Z_{i,k}\cap Z_{i,k+1})\cup   \bigcup_{v\in D_{i+(k+1)r-1}}\sigma^{v^{-1}}C_{i+(k+1)r}
  \end{eqnarray*}
  
Using an induction argument, we get that for every $m>k$
  \begin{eqnarray*}
   Z_{i,k} & \subseteq & \bigcap_{l=k}^m Z_{i,l} \cup \bigcup_{l=k+1}^m\bigcup_{v\in D_{i+lr-1}}\sigma^{v^{-1}}C_{i+lr}.\\
              & \subseteq & \bigcap_{l=k}^m Z_{i,l} \cup \bigcup_{l\geq k+1}\bigcup_{v\in D_{i+lr-1}}\sigma^{v^{-1}}C_{i+lr}.
              \end{eqnarray*}
From this we get (\ref{intersection}).
 \end{proof}

\begin{lemma}\label{Full-measure-mu}

For every $1\leq i\leq r$, for every $k\geq 0$ and for every invariant probability measure $\mu$ of $X$ we have
$$
\mu\left( \bigcap_{l\geq k} Z_{i,l} \right)\leq \mu(Z_{i,k})\leq \mu\left( \bigcap_{l\geq k} Z_{i,l} \right)+\sum_{l={k+1}}^{\infty}\frac{|D_{i+lr-1}|}{|D_{i+lr}|}. 
$$ 
Moreover, if  $p(\mu)=p(\nu_i)$ then 
$$
\lim_{k\to \infty}\mu\left( \bigcap_{l\geq k} Z_{i,l} \right)=1,
$$
and $\mu(A_i)=1$, where $$A_i = \bigcap_{g\in G} \bigcup_{k\geq 0} \bigcap_{l\geq k} \sigma^g (Z_{i,l}).
$$
\end{lemma}
\begin{proof}
The first part of this statement follows directly from Lemma \ref{intersection-lemma}.
Observe that Lemma \ref{crecimiento-2} and \cite[Theorem 28.4]{Kn47} imply that $\lim_{k\to \infty}\sum_{l={k+1}}^{\infty}\frac{|D_{i+lr-1}|}{|D_{i+lr}|}=0$.

If $p(\mu)=p(\nu_i)$ then Lemma \ref{partitions-determine-measures} implies $\mu(Z_{k,i})=\nu_i(Z_{k,i})$ for every $k$. Then Lemmas \ref{full-measure} and \ref{intersection-lemma} imply that $\lim_{k\to \infty}\mu\left( \bigcap_{l\geq k} Z_{i,l} \right)=1$ and this in turn  implies that $\mu\left(\bigcup_{k\geq 0}\bigcap_{l\geq k} Z_{i,l}\right)=1$. Since $\mu$ is invariant, we get that $
\mu(A_i)=1.
$  
\end{proof}
\begin{lemma}\label{Injectivity-pi}
    For $1\leq i\leq r$, let $A_i$ be the set defined in Lemma \ref{Full-measure-mu}. The factor map $\pi\colon X\to \overleftarrow{G}$ from $X$ to its associated $G$-odometer is injective when restricted to $A_i$.
\end{lemma}
\begin{proof}
    Let $x,y\in A_i$ such that $\pi(x)=\pi(y)$.  We have that for each $g\in G$ there exists $k_g\geq  0$ such that $\sigma^{g^{-1}} x,\sigma^{g^{-1}} y\in Z_{i,l}$, for every $l\geq k_g$. Consider $k\geq k_g$. Since $\pi(x)=\pi(y)$, we get that $x,y\in \sigma^{v^{-1}_{i+kr}}C_{i+kr}$ for some  $v_{i+kr}\in D_{i+kr}$ and hence $\sigma^{g^{-1}} x,\sigma^{g^{-1}}y\in \sigma^{(v_{i+kr}g)^{-1}}C_{i+kr}$. If $v\in D_{i+kr}$ is such that   $v_{i+kr}g\in v\Gamma_{i+kr}$, then $\sigma^{g^{-1}} x,\sigma^{g^{-1}}y\in \sigma^{v^{-1}}C_{i+kr}$. Since $\sigma^{g^{-1}} x$ and $\sigma^{g^{-1}}y$ belong to $Z_{i,k}$, we get
     $\sigma^{g^{-1}} x,\sigma^{g^{-1}}y\in \sigma^{v^{-1}}C_{i+kr,i}$. Thus, there exist $w,z\in C_{i+kr,i}$ such that $\sigma^{g^{-1}}x=\sigma^{v^{-1}}w$    and $\sigma^{g^{-1}}y=\sigma^{v^{-1}} z$. Since $w,z\in C_{i+kr,i}$, we have $w(D_{i+kr})=z(D_{i+kr})$, which in turn implies that $w(v)=z(v)$. Consequently, we get 
    \begin{align*}
        x(g)=\sigma^{g^{-1}}x(1_G)=\sigma^{v^{-1}} w(1_G)=w(v)=z(v)=\sigma^{v^{-1}}z(1_G)=\sigma^{g^{-1}} y(1_G)=y(g).
    \end{align*}
    Since this is true for every $g\in G$, we conclude that $x=y$.
\end{proof}

\begin{proposition}\label{conjugacy}
    Let $\mu$ be an invariant probability measure of $X$. If $ \mu(A_i)=1$,  for some  $1\leq i\leq r$, then $(X,\sigma, \mu)$ and $(\overleftarrow{G},\phi,\nu)$ are measure conjugate. 
\end{proposition}
\begin{proof}
   Suppose there exists $i$ such that $\mu(A_i)=1$. By Lemma \ref{Injectivity-pi} we obtain that $\pi|_{A_i}\colon A_i\to \pi(A_i)$ is bijective.
   Using \cite[Theorem 2.8]{Gl} we get that $\pi(A_i)$ is measurable. 
   Moreover, using that $A_i$ is $G$-invariant and $\pi\colon X\to\overleftarrow{G}$ is a closed factor map we can conclude that $\pi|_{A_i}$ is an isomorphism of measurable spaces such that $\pi|_{A_i}(\sigma^g x)=\phi^g\pi|_{A_i}(x)$.
   Note that $\overleftarrow{G}$ is uniquely ergodic and for this reason, $\mu(\pi|_{A_i}^{-1}(B))=\nu(B)$ for every measurable subset $B$ of  $\pi(A_i)$. This implies that $\pi|_{A_i}$ is a measure conjugacy. 
\end{proof}

 \begin{corollary}\label{cor:injectivep}
     Let $p$ be the surjective affine map from $\cM(X,\sigma,G)$ to $\triangle$ introduced in Proposition 4.12. Let $\mu$ be an invariant measure of $(X,\sigma)$. If $p(\mu)=p(\nu_i)$ for some $i\in\{1,2,\dots, r\}$, then $\mu=\nu_i$
 \end{corollary}
 \begin{proof}
 Suppose $p(\mu)=p(\nu_i)$ for some $1\leq i\leq r$. Let $A_i$ be defined as in Lemma \ref{Full-measure-mu}. By Lemma \ref{Full-measure-mu} and Proposition \ref{conjugacy}, $ (X,\mu)$ and $(\overleftarrow{G},\nu)$ are measure conjugate  and $ (X,\nu_i)$ and $(\overleftarrow{G},\nu)$ are also measure conjugate with  measure conjugacy $\pi|_{A_i}$ as in the proof of Proposition \ref{conjugacy}. Let $O\subseteq X$ be an open set. Since $O\cap A_i$ is an open set in $A_i$ and $\pi\mid_{A_i}$ is a measure conjugacy, $\pi\mid_{A_i}(O\cap A_i)$ is a measurable set in $\overleftarrow{G}$. Let $V=\pi\mid_{A_i}(O\cap A_i)$. Since $\mu(\pi|_{A_i}^{-1}(B))=\nu_i(\pi|_{A_i}^{-1}(B))=\nu(B)$, for every measurable subset $B$ of $\pi(A_i)$, we obtain
 \begin{align*}
     \mu(O)=\mu(O\cap A_i)=\mu(\pi|_{A_i}^{-1}(V))=\nu(V)=\nu_i(\pi|_{A_i}^{-1}(V))=\nu_i(O\cap A_i)=\nu_i(O).
 \end{align*}
 Since the open set $O$ was arbitrarily 
 taken, we conclude that $\mu=\nu_i$.
    
 \end{proof}

\begin{proof}[Proof of Theorem \ref{theo:main2}]
Let $1\leq i\leq r$. Suppose that $\nu_i$ is not an extreme point of $\cM(X,\sigma,G)$. This implies that there exist two distinct invariant measures $\mu_1$, $\mu_2$, and $0<t<1$ such that
$$\nu_i=t\mu_1+(1-t)\mu_2.$$
Thus,
$$p(\nu_i)=t p(\mu_1)+(1-t)p(\mu_2).$$
If $p(\mu_1)=p(\mu_2)$, then $p(\nu_i)=p(\mu_2)$, which by Corollary \ref{cor:injectivep} implies that $\mu_1=\nu_i$, and then $\nu_i=\mu_1=\mu_2$, a contradiction. Thus, $p(\mu_1)$ and $p(\mu_2)$ are two distinct elements in $\triangle$ and $p(\nu_i)$ lies between them. But this is not possible since $p(\nu_i)=\vec{t_i}$ is an extreme point in $\triangle$. We conclude that $\nu_i$ must be an extreme point of $\cM(X,\sigma, G)$, {\em i.e.}, an ergodic measure. From Lemma \ref{Full-measure-mu} and Proposition \ref{conjugacy} we get (1) of Theorem \ref{theo:main2}.

For the second part of the statement, note that for every $\mu\in\cM(X,\sigma,G)$, we have that there exist $\alpha_j\in [0,1]$ such that $\sum_{j=1}^r \alpha_j=1$ and $p(\mu)=\sum_{j=1}^r \alpha_j \vec{t_j}$. In particular, 
\begin{align*}
    \mu([i])=\sum_{j=1}^r\alpha_j\nu_j([i])=\alpha_i(1-d)+t_i=\alpha_i\nu_i([i])+(1-\alpha_i)t_i.
\end{align*}
Since $t_i<\nu_i([i])$, we have that $\mu([i])=\nu_i([i])$ if and only if $\alpha_i=1$.

Finally, Lemma \ref{odometers-compactification} allows to write the statement in terms of totally disconnected metrisable compactification of $G$.
\end{proof}

\subsection*{Acknowledgements} The authors warmly thank Raimundo Brice\~no and Samuel Petite for many fruitful discussions.

\end{document}